\documentclass[11pt]{amsart}

\usepackage{amssymb, setspace}
\usepackage[margin=0.6in]{geometry}
\usepackage{graphicx}
\usepackage{xcolor}
\usepackage{hyperref}

\newtheorem{theorem}{Theorem}[section]
\newtheorem{lemma}[theorem]{Lemma}

\newtheorem{proposition}[theorem]{Proposition}
\newtheorem{corollary}[theorem]{Corollary}
\newtheorem{example}[theorem]{Example}
\newtheorem{problem}[theorem]{Problem}
\theoremstyle{definition}

\theoremstyle{remark}
\newtheorem{remark}[theorem]{Remark}
\numberwithin{equation}{section}
\def\na{\nabla}
\newcommand{\e}{{\epsilon}}
\newcommand{\R}{{\mathbb R}}
\newcommand{\str}{\Gamma^*}
\newcommand{\RR}{R}
\newcommand{\W}{\Omega}
\newcommand{\bareta}{\bar{\eta}}
\newcommand{\etaone}{\bar{\eta}}
\newcommand{\barrho}{\bar{\rho}}
\newcommand{\aone}{{a}}
\newcommand{\atwo}{{b}}
\newcommand{\cH}{\mathcal{H}}
\newcommand{\mres}{\mathbin{\vrule height 1.6ex depth 0pt width
		0.13ex\vrule height 0.13ex depth 0pt width 1.3ex}}
		
\author{Mark Allen}
\address[Mark Allen]{Department of Mathematics, Brigham Young University, Provo,  UT}
\email{allen@mathematics.byu.edu}

\author{Dennis Kriventsov}
\address[Dennis Kriventsov]{Department of Mathematics, Rutgers University,  Piscataway, NJ}
\email{dnk34@math.rutgers.edu}

\author{Robin Neumayer}
\address[Robin Neumayer]{Department of Mathematical Sciences, Carnegie Mellon University, Pittsburgh, PA}
\email{neumayer@cmu.edu}

\thanks{
MA was partly supported by the Simons Collaboration Grant 637757. RN was partly supported by NSF Grant DMS-1901427 and by the Gregg Zeitlin Early Career Professorship. Part of this work was accomplished when the first two authors visited the third author at Carnegie Mellon University, when the first two authors visited the Mittag-Leffler Institute, and when the third author visited the Fields Institute. The authors would like to thank Max Engelstein for helpful discussions regarding connections to harmonic analysis.}

\title{Rectifiability and uniqueness of blow-ups for points with positive Alt-Caffarelli-Friedman limit}
\begin{document}

\maketitle
\begin{abstract}
		We study the {regularity of the interface between the disjoint supports of} a pair of nonnegative subharmonic functions. The portion of the interface where the Alt-Caffarelli-Friedman (ACF) monotonicity formula is asymptotically positive forms an $\mathcal{H}^{n-1}$-rectifiable set. Moreover, for $\mathcal{H}^{n-1}$-a.e. such point, the two functions have unique blowups, i.e. {their} Lipschitz rescalings converge in $W^{1,2}$ to a pair of nondegenerate truncated linear functions whose supports meet at the approximate tangent plane. The main tools used include the Naber-Valtorta framework and our recent result establishing a sharp quantitative remainder term in the ACF monotonicity formula. We also give applications of our results to free boundary problems. 
\end{abstract}

\section{Introduction}
Recent decades have seen a significant body of research focused on understanding the interplay between the geometry of subsets of Euclidean space, their analytic properties, and the behavior of solutions to partial differential equations on these sets.
 A central goal of this program is to give geometric conditions that imply or characterize the regularity---often rectifiability---of a given set. 
 Results of this type go back to Reifenberg \cite{Reifenberg}, and since pioneering work in the nineties \cite{Jones90, DavidSemmes91, BJ94, ToroDuke, Leger},
 there has been a steady stream of increasingly refined results with compelling applications. 
 Two  recent and fundamental results in this direction are Jaye, Tolsa, and Villa's resolution to Carleson's $\e^2$-conjecture \cite{JTV21} and Naber and Valtorta's  Rectifiable Reifenberg Theorem \cite{NVmain} (see also \cite{AzzamTolsa, Tolsa, ENV} for related results), both of which give criteria for a set to be rectifiable in terms of geometric square functions that measure flatness at all points and scales.  The latter, and the Naber-Valtorta framework more broadly, have powerful applications to singularity analysis in various nonlinear PDE, such as (approximate) harmonic maps \cite{NVmain, NVApprox}, 
 %$Q$-valued harmonic maps \cite{DMSV}, 
 stationary varifolds \cite{NVVarifolds},
 and free boundary problems \cite{EE19}, and will also play a key role here.
\smallskip
% proving that for a Jordan curve $\Gamma \subset \R^2$, the set of points where a certain multi-scale geometric square function is finite forms an $\mathcal{H}^1$-rectifiable set.  %Another important recent result in this vein is the
    % These geometric assumptions frequently involve multi-scale geometric square functions that encapsulate the local proximity of a set to a $k$-dimensional plane, summed up over all scales.
    % \cite{Reifenberg}, %L^infty betas square summable at all scales and points implies bi-holder to a ball. 
    % Jones  \cite{Jones90},  % A bounded planar set is contained in a rectifiable curve iff the L^infty beta numbers are square summable at all scales and points; this is where beta numbers were "introduced", except that Reifenberg used them already.
%David-Semmes \cite{DavidSemmes91}, % For d-dim Ahlfors regular sets in R^n,  they give ~5 equivalent conditions, including the L^1 beta numbers giving a Carleson measure (geometric condition) and the set containing big pieces of Lipschitz graphs. This is also where non-L^infty beta numbers were used for the first time.
%Bishop-Jones \cite{BJ94},
% Toro \cite{ToroDuke}, %A set that satisfies a stronger type of Reifenberg condition is bi-Lipschitz to an open subset of R^n
%Leger \cite{Leger}, % For a curve in R^n, if the Menger curvature is finite then the curve is rectifiable.
    
In this paper we give a different kind of rectifiability criterion, which relates the boundary behavior of nonnegative subharmonic functions to the regularity of a set.  The key quantity is the {\it Alt-Caffarelli-Friedman (ACF) monotonicity formula}
\begin{equation}\label{eqn: ACF formula}
	J_x(r) = \left( \frac{1}{r^2} \int_{B_r(x)} \frac{|\na u|^2}{|{x-y}|^{n-2}}\right)  \,\left(  \frac{1}{r^2} \int_{B_r(x)} \frac{|\na v|^2}{|{x-y}|^{n-2}} \right)
\end{equation}
for a pair of nonnegative subharmonic functions $u$ and $v$ with $u\cdot v=0$. Our first main result, Theorem~\ref{thm: main rectifiability} below, shows that the set $\str$ of  points $x$ in the  interface $\Gamma:=\partial\{u>0 \}\cup \partial \{v>0 \}$   where $\lim_{r\to 0}J_x(r)$ is positive forms an $\cH^{n-1}$-rectifiable set. Examples show the conclusion is in a sense optimal: this set does not need to have an approximate tangent space at {\it every} point, and the full interface need not be rectifiable.
Our second main result, Theorem~\ref{thm: uniqueness of blowups}, concerns the boundary behavior of the functions $u,v$ themselves: at $\cH^{n-1}$-a.e. point $x$ in $\str$, the functions have unique blowups to a pair of truncated linear functions. Our hypotheses are unconditional in that they require no a priori information about the interface, and the asymptotic positivity of $J$ is a condition that is easily checked in the context of free boundary problems. These two key features allow us to apply our result to a broad class of free boundary problems, which we outline in section~\ref{ssec: FBP}.
\smallskip

More precisely,  fix $n\geq 2$ and let $u, v :B_{10}(0)\to \R$ be nonnegative continuous functions  that satisfy
\begin{equation}\label{eqn: ACF setup}
\begin{cases}
	-\Delta u \leq 0 & \text{ in } \{ u>0\},\\
		-\Delta v \leq 0 & \text{ in }  \{ v>0\},\\
	u\cdot v =0 & \text{ in } B_{10}(0).
\end{cases}
\end{equation}
The simplest example of such a pair of functions is a pair of truncated linear functions supported on complementary half-planes meeting at a hyperplane interface, i.e.  
\begin{equation}
	\label{eqn: truncated linear functions}
u(y) = \aone((y-x)\cdot \nu)^+\qquad \text{ and }\qquad v(y) = \atwo ((y-x)\cdot \nu)^-\,
\end{equation}
for $x \in \R^n,$ constants $\aone, \atwo>0$ and $\nu \in \mathbb{S}^{n-1}.$ More generally, however, the condition   \eqref{eqn: ACF setup} is quite flexible and allows for a rich assortment of behaviors from $u,v$, and the interface between their supports; see the examples in section~\ref{ssec: examples}. One encounters the configuration \eqref{eqn: ACF setup} in various settings, for instance when $u$,$v$ are phases of a solution to the two-phase Bernoulli problem in models for jets and cavities \cite{acf84, acfmodel1, acfmodel2}, when $u$ and $v$ are the positive and negative parts of directional derivatives of solutions to obstacle-type problems \cite{PetrosyanBook, Shahgholian03}, or when $u$,$v$ are population densities for two species in segregation models \cite{CaffSeg, TerraciniSeg}. In all of these contexts, the regularity of the interface 
$\Gamma$
%$\partial\{ u>0\} \cup \partial \{v>0\}$
 is a priori unknown, and indeed a typical objective is to understand its regularity.  From the viewpoint of geometric measure theory and harmonic analysis, it is natural to use the boundary behavior of (sub-)solutions to PDE as a way to understand the regularity of an interface; this is typically studied via harmonic and elliptic measure \cite{HarmMeas1, HarmMeas2, ElliptMeas, ElliptMeas2}.
\smallskip
 
The ACF monotonicity formula \eqref{eqn: ACF formula} was introduced by Alt, Caffarelli, and Friedman in \cite{acf84} and since then has found a wide variety of applications, including to the settings described in the previous paragraph.
%\footnote{Add some citations here? We can re-cite the ones above or is that redudant? In you guys' counterexample paper, you cite some of Mark's papers here, let's add that. In our QACF paper we state that the ACF formula also has applications in harmonic analysis. Is this true? Do we have anything to cite here? \edit{I think the paragraph above already discusses this sufficiently, can just say that you can apply it to any of the examples there.}} 
%
The key feature of the quantity \eqref{eqn: ACF formula} is its monotonicity: given $u,v$ as above and $x \in B_4(0)$, the function $r \mapsto J_x(r)$ is nondecreasing for $r \in (0,4)$. 
Moreover,  if $J_x(r_1) = J_x(r_2)$ for $r_1<r_2$, then either one of $u,v$ is {identically} zero in $B_{r_2}(x)$ (in which case $J_x(r) \equiv 0$), or else $u$ and $v$ are a pair of complementary truncated linear functions as in \eqref{eqn: truncated linear functions} in $B_{r_2}(x)$; see \cite{CaffSalsa}.
\smallskip

This rigidity statement contains stronger information when $J_x(r_1)>0$, and in particular implies that the interface $\Gamma$ is locally a hyperplane. Instead if $J_x(r_1)=0$, the fact that $u$ or $v$ vanishes identically does not impose any interfacial structure. It is natural, then, that the subset of the interface $\Gamma$ where one expects good structure is the set of points where the ACF formula is asymptotically positive at small scales, that is, on
 	\begin{equation*}
\label{eqn: regular part}
	\Gamma^*  := \left\{ x  :  J_x(0^+) >0 \right\}.
		\end{equation*}
Here	$ J_x(0^+): = \lim_{r \to 0^+} J_x(r)$; monotonicity guarantees the existence of the limit.
The basis of this paper is to exploit {\it almost }rigidity  in the ACF monotonicity formula via our results in \cite{AKN1} to prove structural properties of $\str$ and interfacial regularity of $u$ and $v$ on this set. Our first theorem asserts the $\cH^{n-1}$-rectifiability of $\str$.
\begin{theorem}[Rectifiability]\label{thm: main rectifiability}
Fix $n\geq 2$ and let $u, v :B_{10}(0)\to \R$ be nonnegative continuous functions that satisfy \eqref{eqn: ACF setup}. Then $\Gamma^*\cap B_1(0)$ is $\mathcal{H}^{n-1}$-rectifiable.
\end{theorem}
We will see examples in section~\ref{ssec: examples} demonstrating that on the remaining portion of the interface where $J_x(0^+)=0$, the interface can fail to have an approximate tangent plane at {\it every } point, even under the stronger assumptions that $u$ and $v$ are harmonic where they are positive.  Furthermore, $\str$ is not better than rectifiable, and in particular the assumption $J_x(0^+)>0$ does not imply the existence of an approximate tangent plane at $x$.
 Our second main theorem shows that the functions $u$ and $v$ have unique blowups at $\mathcal{H}^{n-1}$-a.e. point in $\Gamma^*$.

%In the context of two-phase free boundary problems, the functions $u$ and $v$ are the positive and negative parts of a solution and additionally satisfy a transmission condition across $\Gamma$. The transmission condition together with solving an equation often forces $\Gamma$ to be rectifiable (and often better). However, without a transmission condition, the rectifiability of $\Gamma$ is simply false in general, as we can see in the examples below.\footnote{This paragraph is terrible.}

%In turns out, as we show in Theorem~\ref{thm: main rectifiability}, that a natural subset of the interface is identified via the {\it Alt-Caffarelli-Friedman monotonicity formula}.

\begin{theorem}[Uniqueness of blowups]\label{thm: uniqueness of blowups}
Fix $n\geq 2$ and let $u, v :B_{10}(0)\to \R$ be nonnegative continuous functions that satisfy \eqref{eqn: ACF setup}. 
	For $\cH^{n-1}$-a.e. $x \in \Gamma^* \cap B_{1/2}(0)$, there exist constants $\aone^x, \atwo^x > 0$ and a unit vector $\nu^x \in \mathcal{S}^{n-1}$ such that the rescaled functions
	\begin{equation}
		\label{eqn: lipschitz rescalings}
		u^{x, r}(z) = \frac{u(x + r z)}{r} \qquad v^{x, r}(z) = \frac{v(x + r z)}{r}
		\end{equation}
	converge locally in the strong $W^{1,2}$ topology
	 to $\ell_+(z) = \aone^x (z \cdot \nu)^+$ and $\ell_-(z) = \atwo^x (z \cdot \nu)^-$, respectively. Moreover, $\aone^x\atwo^x = c_*J_x(0^+)$, where $c_* > 0$ is an explicit dimensional constant.
\end{theorem}

\subsection{Discussion of the proofs}
We use two main tools to prove Theorem~\ref{thm: main rectifiability}. The first is Naber-Valtorta's quantitative stratification and  rectifiable Reifenberg framework. Introduced in \cite{NVmain}, these techniques are designed to study singular sets in geometric PDE and are remarkably flexible:  whenever one has a solution to a PDE with (a) a monotonicity formula with a quadratic remainder term, and (b) strong compactness properties for sequences of solutions (usually coming from a priori estimates),  one can apply the Naber-Valtorta tools to prove rectifiability of the strata and estimates for the quantitative strata. If one additionally has an $\e$-regularity theorem,  this automatically gives packing and Hausdorff measure estimates for the full singular set.
\smallskip

The  present context does not fit within this standard setup. First, we do not consider solutions to a PDE, at least not in a conventional sense: in \eqref{eqn: ACF setup} we assume only that  $u$ and $v$ are {\it sub}solutions to the Laplace equation and impose no sort of transmission condition across the interface as one would in a free boundary problem. While subharmonicity provides important structure, the setup in \eqref{eqn: ACF setup} is simply too weak of an ``equation'' to expect a real regularity theory. Nonetheless,  functions satisfying \eqref{eqn: ACF setup} do enjoy a monotonicity formula, and this suggests some hope of applying the Naber-Valtorta ideas. However, until our recent  results in \cite{AKN1}, both key ingredients (a) and (b) described above were entirely missing in this context. 
\smallskip

The second key tool used to prove Theorem~\ref{thm: main rectifiability}, which equips us with ingredient (a),
is the sharp quantitative remainder term in the ACF monotonicity formula we proved in \cite{AKN1}; see Theorem~\ref{t:ACF} below. 
Often,  a monotonicity formula's proof automatically yields a quantitative remainder term that detects how far an object is from a ``cone'' solution; this is the case, for instance, for the area ratio for minimal surfaces,  the normalized energy for harmonic maps, and the Weiss formula for Bernoulli free boundary problems.
Instead, it is quite challenging to glean from the proof of the ACF formula how far functions $u,v$ satisfying \eqref{eqn: ACF setup} are from truncated linear functions as in \eqref{eqn: truncated linear functions}. The quantitative remainder term in \cite{AKN1} uses  a new type of sharp quantitative stability estimate for the Faber-Krahn inequality on the sphere, which in turn relies on delicate free boundary regularity results in \cite{AKN2}. 
Ingredient (b) in this context is false in the strength one might hope for, but again it is Theorem~\ref{t:ACF} that provides us with enough ($W^{1,2})$ compactness to adapt the Naber-Valtorta ideas.
\smallskip
% It is essential to the proof of Theorem~\ref{thm: main rectifiability}that our quantitative Faber-Krahn inequality has the optimal (quadratic) exponent. 
\smallskip

Another difference between more typical  applications of the Naber-Valtorta framework and the present setting is that points in $\str$ are not singular points. In classical singularity analysis,  one stratifies the singular set for a solution to a given PDE according to the number of symmetries of its blowups.
On the other hand, for $x \in \str$, every blowup takes the form \eqref{eqn: truncated linear functions}  and in particular is  ``$(n-1)$-symmetric.'' This allows for various simplifications. Most notably, in Section~\ref{sec: rect}, we give a covering argument that is significantly more streamlined than the delicate good tree/bad tree constructions used in other settings.
Throughout the proof of Theorem~\ref{thm: main rectifiability}, we will attempt to emphasize which steps are standard applications of Naber-Valtorta machinery and which steps involve novel ideas or significant departures  from the usual approach. 
As far as the global structure of the Naber-Valtorta framework, we have in particular followed the presentation in \cite{EE19}.
\smallskip

Theorem~\ref{thm: uniqueness of blowups} does not have an analogue in the Naber-Valtorta theory,  and uniqueness of blowups is {generally a} subtle question. 
%The proof of Theorem~\ref{thm: uniqueness of blowups} is based on geometric measure theoretic arguments, using rectifiability of Thoerem~\ref{thm: main rectifiability}, the quantititative estimates on Theorem~\ref{thm: main estimates} below, and the precise form \eqref{eqn: truncated linear functions} of blowups for the Alt-Caffarelli-Friedman monotonicity formula, and the relationship between the blowups and the density of the distributional Laplacian.
Theorem~\ref{thm: uniqueness of blowups} is proven through geometric measure theoretic arguments, using the rectifiability of $\str$ and the measure estimates of Theorem~\ref{thm: main estimates} below together with the precise form \eqref{eqn: truncated linear functions} of blowups and the quantitative  ACF monotonicity formula from \cite{AKN1}. 
 For any sequence $r_k \to 0$, the rescaled functions \eqref{eqn: lipschitz rescalings} at $x\in \str$ subsequentially converge  to a pair of truncated linear functions as in \eqref{eqn: truncated linear functions}. A priori, the slopes $\aone$ and $\atwo$ and the direction $\nu$ might depend on the sequence of scales --- and  for some points they do, as we will see in Example~\ref{ex: Spiral} below. Using Theorem~\ref{thm: main rectifiability}, we show that $\nu$ is determined up to a sign at $\cH^{n-1}$-a.e. point by the unique approximate tangent plane to $\str$. To prove the a.e. uniqueness of the slopes $\aone, \atwo,$ we use the differentiation theory for measures and Theorem~\ref{thm: main estimates} to show that the distributional Laplacians $\Delta u$ and $\Delta v$ have densities with respect to $\str \mres \cH^{n-1}$ at $\cH^{n-1}$-a.e. point in $\str$. The quantitative form of the ACF formula then allows us to relate these densities (which are independent of the blowup sequence) to the {slopes $a, b$}.

\subsection{Quantitative estimates}

Following \cite{CN, NVmain}, we split $\str$ into smaller pieces on which we can obtain quantitative control. Since the ACF formula becomes degenerate when $J(0^+)=0$, we divide $\str$ into pieces where $J_x(0^+)$ is {\it uniformly} positive, and then enlarge these pieces to include points where $J_x(r)$ is uniformly positive down to a definite scale. More specifically, let
$u, v :B_{10}(0)\to \R$ be nonnegative continuous functions that satisfy \eqref{eqn: ACF setup}.
	Fix $\e>0$ and $r>0$, and define the sets
	\begin{align*}
	 \str_\e  = \{ x\in \str \ : \ J_x(0^+) \geq \e\}\,,  & \qquad  \text{ and }\qquad  \str_{\e,r} = \{x \in \str \ : \ J_x(r) \geq \e  \}\,.
	\end{align*}	
Observe that $\str_\e = \bigcap_{r>0} \str_{\e, r} $ and $\Gamma^* = \bigcup_{\e >0} \str_\e$, and that $
	\str_{\e, r} \subset \str_{\delta, s}$ if $\delta < \e \text{ and }r<s.$
 Following \cite{NVmain}, we prove the following  main quantitative estimates that will lead to Theorem~\ref{thm: main rectifiability}.
\begin{theorem}\label{thm: main estimates}
Fix $n\geq 2,$  $\e >0$, and $\bar{J}_0 >0$. 
There is a constant $C_0 = C_0(n,\e, \bar{J}_0)>0$ such that the following holds. 
	Let $u, v :B_8(0)\to \R$ be nonnegative continuous functions satisfying \eqref{eqn: ACF setup} with $0 \in \Gamma(u, v)$ and $\sup_{x \in B_4} J_x(4)\leq \bar{J}_0$. For every $r \in (0,1]$, we can find a finite collection of balls $\{B_r(x_i)\}_{i=1}^N$ such that 
	\begin{equation}\label{eqn: main cover conclusion}
			\str_{\e ,r} \cap B_1(0) \subset \bigcup_{i=1}^N B_r(x_i) \qquad \text{with}  \qquad N \leq C_0 \, r^{1-n}\,.
	\end{equation}
	In particular,  for every $r \in (0,1]$,
	\begin{equation}
		\label{eqn: eps r mink bound}
	\left|B_r(\str_{\e, r}) \cap B_1(0) \right| \leq C_0\, r \qquad \text{and} \qquad 	\mathcal{H}^{n-1} (\str_\e  \cap B_1(0) ) \leq C_0\,.
		\end{equation}
	\end{theorem}

%\noindent\textit{Acknowledgements:}
%MA was partly supported by the Simons Collaboration Grant 637757. RN was partly supported by NSF Grant DMS-1901427 and by the Gregg Zeitlin Early Career Professorship. Part of this work was accomplished when the first two authors visited the third author at Carnegie Mellon University, when the first two authors visited the Mittag-Leffler Institute, and when the third author visited the Fields Institute. The authors would like to thank Max Engelstein for helpful discussions regarding connections to harmonic analysis.
\section{Examples, Applications, and Further Connections}
In this section, we present examples demonstrating the sharpness of our main results, discuss connections to harmonic analysis, and give applications to free boundary problems.

\subsection{Examples}\label{ssec: examples}
Let us give some examples of different pairs of functions $u$ and $v$ satisfying \eqref{eqn: ACF setup} showing what sort of behavior is and is not possible for such functions and their interfaces. Although given in two dimensions, the examples extend to higher dimensions simply by defining $\tilde{u}(x_1, \ldots, x_n) =u(x_1,x_2)$ and $\tilde{v}(x_1, \ldots, x_n) =v(x_1,x_2)$.
The first example shows that Theorem~\ref{thm: main rectifiability} cannot be strengthened to say that the full interface $\partial \{ u>0\} \cup\partial\{v>0\}$ is rectifiable, even with the strengthened assumptions that $u,v$ are {harmonic} in their supports.
\begin{example}[Koch Snowflake]\label{ex: Koch snowflake}
{\rm 
Let $\Gamma \subset \R^2$ be the Koch snowflake, a Jordan curve that is not rectifiable, and let $\Omega$ be the bounded region it encloses. By the {Riemann} mapping theorem and the conformal invariance of the Laplacian in $\R^2$, we can find positive continuous harmonic functions $u$ on $(B_{10} (0)\cap \{x_2> 0\}) \cap \Omega$ and $v$ on $(B_{10} (0)\cap \{x_2> 0\})\setminus \Omega$ that both vanish, say, on the set $\Gamma\cap \{x_2>0\}$. The interface $\partial \{u>0\} \cup \partial\{ v>0\}$ is equal to $\Gamma \cap \{x_2>0\}$. This set is not $\mathcal{H}^1$-rectifiable and in fact {\it no} point on $\partial \{u>0\} \cup \partial \{v>0\}$ has an approximate tangent. %\textcolor{blue}{This example may be extended to higher dimensions by writing $\R^n=\R^2 \times \R^{n-2}$ and letting $\tilde{u}(x_1, \ldots, x_n) =u(x_1,x_2)$ and $\tilde{v}(x_1, \ldots, x_n) =v(x_1,x_2)$ which has the interface $\Gamma \times \mathbb{R}^{n-2}$.}
}\end{example}
\begin{remark}
	{\rm By Theorem~\ref{thm: main rectifiability} we see that the set $\str =\{ J_x(0^+)>0\} $ has $\mathcal{H}^{1}$-measure zero in Example~\ref{ex: Koch snowflake} above. In fact, we can see more directly that $J_x(0^+)=0$ for {\it every} $x \in \Gamma \cap \{x_2>0\}$. Indeed, if it were the case that $J_x(0^+)>0$ at some $x \in \Gamma \cap \{x_2>0\}$, then using the rigidity in the ACF monotonicity formula, we could deduce that every sequence of scales $r_k \to 0$ has a further subsequence along which the functions $u^{x,r_k}$ and $v^{x, r_k}$ converge to a pair of complementary truncated linear functions as in \eqref{eqn: truncated linear functions}. On the other hand, choosing a point $x$ and sequence $r_k$ along which $(\Gamma-x)\cap B_{r_k}(0)/r_k$ looks identical for all $k$, we see this cannot happen.
	}
\end{remark}
%An analogous example can be constructed in higher dimensions.
%\begin{example}
%	{\rm 
%	For $n\geq 2$, write $\R^n=\R^2 \times \R^{n-2}$ and again let $\Gamma \subset$ be the Koch snowflake and consider $\tilde{\Gamma} = \Gamma \times R^{n-2}$; this set is not $\cH^{n-1}$-rectifiable and in fact nowhere has a tangent plane From the self similarity and the Wiener criterion, we may find positive continuous harmonic functions $u,v$ on either side of $\tilde{\Gamma}\cap B_1(0)$ whose zero set in $B_{1}(0)$ is $\tilde{\Gamma}$. 
%	}
%\end{example}
This example also demonstrates how the vanishing of one of the functions $u,v$ does not impose any structure on the interface:
\begin{example}
	{\rm 
Let $\Omega$ and $u$ be as in Example~\ref{ex: Koch snowflake} and let $v$ be identically equal to zero. The ACF monotonicity formula is constantly equal to zero at every point and scale, but as we saw above $\partial\{ u>0\}$ does not have an approximate tangent at any point.
	}
\end{example}

The next example demonstrates that the existence of an approximate tangent plane at $x$ does not imply that $J_x(0^+)>0$. 
\begin{example}\label{ex 2.4}
     {\rm
Let $f: \mathbb{R} \to \mathbb{R}$ be $C^1$, symmetric, nonnegative, increasing for $x>0$, and $f(0)=0$. Let $u$ be positive and harmonic in $\Omega=\{(x,y) \in B_1 \mid y>f(x)\}$ and vanishing on the boundary $\Gamma=\{(x,y) \in B_1 \mid y=f(x)\}$. Similarly, let $v$ be positive and harmonic in $B_1 \setminus \Omega$ and vanishing on $\Gamma$. Then $\Gamma$ will have a tangent at every point. 
A basic computation (see \cite{AKN1}) shows that
\[
	(\log J_0)'(r) \geq \frac{1}{r}[\alpha_u(r) + \alpha_v(r) - 2],
\]
where $\alpha_u(r),\alpha_v(r)$ are the characteristic constants (see Section \ref{s:carleson}), and in this example 
$\alpha_u(r)=\pi-2\arctan(f(r)/r)$ and $\alpha_v(r)=\pi+2\arctan(f(r)/r)$. For $r$ small, we then obtain 
\[
 (\log J_0)'(r) \geq \frac{1}{r}[\alpha_u(r) + \alpha_v(r) - 2] \geq \frac{ [\arctan(f(r)/r)]^2}{2r}. 
\]
By choosing $f(r)/r \to 0$ but slowly enough so that $[\arctan(f(r)/r)]^2$ is not Dini-integrable, it follows that $J_0(0^+)=0$ despite a tangent for $\Gamma$ at the origin. 
%In this example, $J_0(0^+)=  c \lim_{r \to 0} r^{-2} \omega_1(\Gamma \cap B_r) \omega_2 (\Gamma \cap B_r)$ for some constant $c$, and where $\omega_1, \omega_2$ are harmonic measures for $\Gamma$ from $u,v$ respectively.  Furthermore, using Beurling's estimate (as shown in \cite{BJ94}), we have 
%\[
 %\frac{\omega_1(\Gamma \cap B_r) \omega_2 (\Gamma \cap B_r)}{r^2} \leq C_1 \text{exp} \left(-C^2  \int_r^1 \epsilon^2(0,t) \frac{dt}{t}\right),
%\]
%where $\epsilon(0,t)=2\arctan(f(t)/t)$. (The quantity $\epsilon(x,t)$ is defined in section \ref{s:carleson}). By choosing $f(t)/t \to 0$ but slowly enough so that $\epsilon^2(0,t)$ is not Dini-integrable, it follows that $J_0(0^+)=0$ despite a tangent for $\Gamma$ at the origin. 
}
\end{example}

The next example is a converse to the example above, was constructed by the first two authors in  \cite{ak20},  and shows that in 
Theorem~\ref{thm: main rectifiability}, the existence of  an approximate tangent plane at $\cH^{n-1}$-a.e. point in $\str$ cannot be improved to  an approximate tangent at {\it every} point in $\str$. 
\begin{example}[Spiraling interface with $J(0^+)>0$] \label{ex: Spiral}
{\rm
In \cite{ak20}, the first two authors construct a pair of continuous, nonnegative harmonic functions $u,v$ in the plane with disjoint positivity sets such that  $J_0(0^+)>0$ but the interface $\Gamma = \partial \{ u>0\} \cup \partial \{ v>0\}$ does not admit an approximate tangent at $0$. In this example $\Gamma$ is a spiral; in any annulus $B_1(0)\setminus B_r(0)$, $\Gamma$ is a piecewise smooth connected hyperplane, but it spirals in such a way that the closest pair of complementary truncated linear functions \eqref{eqn: truncated linear functions} at scale $r$ has orientation $\nu(r)$ that rotates and does not converge at $r\to 0$.
}
\end{example}
\begin{remark}
	{\rm
	In  \cite[Corollary 1.4]{AKN1}, we give a pointwise criterion ruling out behavior as in Example~\ref{ex: Spiral} at a point $x \in \str$: if $\sum_j \log(J_x(2^{-j})/J_x(0^+))^{1/2} <+\infty$, then the functions $u,v$ have unique of blowups and $\str$ has an approximate tangent at $x$. This criterion asks for {\it very} small energy drop at each dyadic scale, whereas $J_x(0^+)>0$ imposes the weaker condition $\sum_j \log(J_x(2^{-j})/J_x(2^{-j-1})) < +\infty$ on the dyadic energy drop.
	}
\end{remark}

Unlike the study of singular points, when an $\epsilon$-regularity result allows one to pass estimates for a small enough $\epsilon$ to the entire top stratum of singular points, in our situation one cannot expect to pass the measure estimates of Theorem \ref{thm: main estimates} for a small fixed $\epsilon$ to all of $\str$.
The next example illustrates that such estimates are not available for all of $\str$.
\begin{example}
 {\rm
 We modify the construction of the Koch snowflake. The heuristic idea is to retain line segments while continuing the construction on others. In this manner, we obtain a boundary that contains straight line segments whose union is infinite in length. Specifically,  in the first iteration of the snowflake we have an equilateral triangle whose sides have length one, so the total length of the triangle is $3$ which is greater than $2$. We retain two line segments whose union we label as $R_1$ (which has total length $2$ which is greater than $1$), and continue the Koch construction on the third line segment. We run enough Koch iterations on the third line segment until the length is greater than $2$, and label this piece $K_2$, so that the resulting figure is $R_1 \cup K_2$. We choose only one line segment of $K_2$ on which to continue the Koch iterations, and label the union of the remaining line segments we retain as $R_2$ (which has total length greater than $1$). We inductively continue in the same manner as follows: suppose $K_i$ and $R_i$ have been chosen. We run enough Koch iterations on the one line segment of $K_i\setminus R_i$ to obtain $K_{i+1}$ with $\mathcal{H}^{1}(K_{i+1}) > 2$. We then choose one line segment of $K_{k+1}$ and retain the union of the remainder of the line segments as $R_{i+1}$. Continuing inductively, we obtain a limiting figure which we label $\Gamma$. We define $\Omega, u, v$ similarly to Example~\ref{ex: Koch snowflake} (after possibly rotating $\Gamma$). We note that $\mathcal{H}^1(R_i)>1$  for each $i$, and each $R_i$ consists of the union of finitely many line segments. The interior of each line segment is contained in $\Gamma^*$. Also, $\mathcal{H}^{1}(R_i)>1$, and $\mathcal{H}^{1}(R_i \setminus \Gamma^*)=0$. Since $\mathcal{H}^{1}(R_i \cap R_j)=0$ for $i \neq j$ and since there are infinitely many $R_i$, it follows that $\mathcal{H}^{1}(\Gamma^*)=\infty$.  
 }
\end{example}

\subsection{Connections to harmonic analysis, Bishop's conjecture, and Carleson's $\e^2$ problem} \label{s:carleson}
A problem in harmonic analysis closely related to Theorem \ref{thm: main rectifiability} is known as Bishop's conjecture. Given a pair of complementary domains $\W^\pm$, let $\omega^\pm$ be the harmonic measure on $\partial \W^\pm$ (with pole somewhere in $\W^\pm$). In \cite{Bishop} Bishop proved that in the plane, if $\omega^+ \ll \omega^- \ll \omega^+$ on a set $E \subset \partial \W^+ \cap \partial \W^-$, then $\omega^\pm$ are mutually absolutely continuous with respect to ($1$-dimensional) Hausdorff measure on a subset $F \subset E$ with $\omega^\pm(E \setminus F)=0$, and $F$ is $1$-rectifiable. A sequence of papers \cite{KPT,AMT,AMTV} proved similar statements in arbitrary dimensions with progressively weaker assumptions on $\W^\pm$, with \cite{AMTV} removing all assumptions and fully establishing the conjecture.

At least if the $u, v$ in $J_x$ are a pair of Green's functions, the hypothesis that $J_x(0^+) > 0$ on $E$ is different and to some extent stronger than the assumption that $\omega^\pm$ are mutually absolutely continuous. Our conclusion is also stronger in the sense that we show that the full set $E$ is  $\cH^{n-1}$ rectifiable, a conclusion which requires stronger assumptions than those in Bishop's conjecture. The proofs of the results in \cite{KPT,AMT,AMTV}, while also exploiting the monotonicity of $J$, are based on deep connections between rectifiability and Riesz transforms, and are very different from the method here. Our argument is more elementary and direct, which potentially leads to more quantitative information (such as Theorem \ref{thm: main estimates}).

Also studied in the harmonic analysis literature is the ``converse'' Bishop's conjecture: if one knows that $E$ is  $\cH^{n-1}$-rectifiable, are $\omega^\pm$ mutually absolutely continuous on $E$? When $n \geq 3$, some mild density assumptions on $\W^\pm$ are imposed, and ``rectifiable'' is replaced by the existence of strong tangents, this is shown in \cite{AMT} (when $n = 2$ it is contained in \cite{Bishop}). It is unclear to what extent this resolves the problem converse to Theorem \ref{thm: main rectifiability}, namely:
\begin{problem}\label{p1} Assume that $E = \partial \{u > 0\} \cap \partial \{v > 0\}$ is  $\cH^{n-1}$ rectifiable. Is $J_x(0^+) > 0$ for $\cH^{n-1}$-a.e. $x \in E$ (under some minimal assumptions on $E$ or $u, v$)?
\end{problem}
Example~\ref{ex 2.4} shows that in Problem~\ref{p1}, $J_x(0^+)$ need not be positive for {\it every} $x \in E$.
From a slightly different perspective: Bishop and Jones \cite{BJ94} proved the following striking result: for a suitably nice domain $\Omega \subset \R^2$ (e.g. one bounded by a Jordan curve), the set of points for which
\[
	\beta_\infty(x, r) = \inf_{\substack{L \text{ a line}, \\ L \cap B_r(x) \neq \varnothing}} \sup_{x \in \partial \W \cap B_r(x)} \frac{d(x, L)}{r}
\]
satisfies the square-Dini condition
\[
	\int_0^1 \frac{\beta_\infty^2(x, r)dr}{r} < \infty
\]
is precisely the rectifiable portion $\partial \W$ (up to sets of $\cH^1$ measure $0$). They pose a more general question along these lines: for which choices of geometric ``square function'' (to use the term loosely, as in \cite{JTV21}) in place of $\beta_\infty$ (as well as dimension and \emph{a priori} assumptions on $\partial \W$) does one recover such a rectifiability criterion? There has been substantial study of this topic (\cite{Tolsa}), and the rectifiable Reifenberg theorems of \cite{NVmain} fit naturally into this program. A particularly striking recent result of \cite{JTV21} is that in two dimensions $\beta_\infty$ can be replaced by the quantity $\e(x, r)$, defined as follows. Let $I^+(x, r)$ length of the longest arc contained in $\W \cap B_r(x)$, and $I^-(x, r)$ the length of the longest arc contained in $B_r(x) \setminus \bar{\Omega}$; then $\e(x, r) = \max\{|I^+(x, r)/r - \pi|, |I^-(x, r)/r - \pi|\}$. This answered a question of Carleson explicitly left open in \cite{BJ94}.

The quantity $\e$ is, of course, purely geometric, but we would like to suggest a spectral reinterpretation of it which exposes certain connections to Theorem~\ref{thm: main rectifiability} as well as to higher-dimensional versions of Carleson's problem. A basic computation (see \cite{AKN1}) shows that
\[
	(\log J_0)'(r) \geq \frac{1}{r}[\alpha_u(r) + \alpha_v(r) - 2] \geq \frac{1}{r}\left[|\alpha_u(r) - 1|^2 + |\alpha_v(r) - 1|^2\right],
\]
where $\alpha_u$ represents the homogeneity of the unique nontrivial nonnegative homogeneous function on a cone $G$ with cross-section $G\cap \partial B_r(0) =  \{ u > 0 \}\cap \partial B_r(0)$ and which vanishes on the boundary of $G$, the ``characteristic constant.'' This number can be computed directly from the first Dirichlet eigenvalue of $\{u > 0\} \cap \partial B_r(0)$, which specifically in 2D depends only on the numbers $I^\pm(0, r)$. In particular, we have that
\[
	\frac{1}{r}\left[|\alpha_u(r) - 1|^2 + |\alpha_v(r) - 1|^2\right] \geq \frac{c}{r} \e^2(0, r).
\]
%if $\{u > 0\} = \W$ and $\partial\{v > 0\} = \partial \W$. 
Now, $(\log J_x)'$ is integrable in $r$ if $J_x(0^+) >0$, meaning $(\log J)'$ is a natural (and stronger) square function. If $\partial \W$ is a Jordan curve (or under some other assumptions, still in 2D), then the results of \cite{JTV21} and the argument sketched above imply Theorem \ref{thm: main rectifiability}.

The main interest of Theorem \ref{thm: main rectifiability}, though, is that it holds unconditionally and in any dimension, giving a rectifiability criterion which is analytic in nature. It is clear from the heuristic presented above that $J_x(0^+) > 0$ is a stronger property than the square-summability of $\e$: indeed, a more careful analysis of $(\log J)'$ shows that it provides some mild indirect control over how $\alpha_u(r)$ changes as a function of $r$, something that $\e$ notoriously does not. On the other hand, this suggests that of the many possible generalizations of $\e$ to higher dimensions, a particularly compelling one is the ``spectral'' square function:
\[
	\lambda^2(x, r) = |\lambda_1(\W \cap \partial B_r(x))/r^2 - (n - 1)|^2 + |\lambda_1(\partial B_r(x)\setminus \W)/r^2 - (n - 1)|^2,
\]
where $\lambda_1$ stands for the first Dirichlet eigenvalue of a domain on a sphere ($n-1$ is $\lambda_1$ of a half unit sphere).
\begin{problem}
	Let $\W \subset \R^n$ be a connected open set. Under what minimal assumptions on $\partial \W$ does
	\[
		\left\{x \in \partial \W\ : \ \int_0^1 \frac{\lambda^2(x, r)}{r}dr < \infty \right\}
	\]
	coincide with the rectifiable portion of $\partial \W$, up to sets of $\cH^{n-1}$ measure 0?
\end{problem}

\subsection{Applications to free boundary problems}\label{ssec: FBP}

In this section we present informally an example of how one might apply Theorem \ref{thm: main rectifiability} to gain information about a free boundary problem. We will consider a broad class of  vector-valued Bernoulli problems in \eqref{eqn: FBP} below. Problems like this derive from eigenvalue optimization and segregation models, for instance, and have been studied in the literature only very recently. The theory is still incomplete and progress so far has been limited to the restrictive setting of the variational formulation. Here we apply the main results of this paper to analyze part of the free boundary without resorting to variational techniques or any special structure of the free boundary condition.

Suppose a domain $\W$ and a collection $\{ u_i\}_{i=1}^N$ of continuous functions $u_i : \bar{\Omega} \rightarrow \R$ satisfy the overdetermined system 
\begin{equation}
	\label{eqn: FBP}
\begin{cases}
	\Delta u_i(x) = 0 & x \in \W \cap B_1\\
	u_i(x) = 0 & x \in \partial \W \cap B_1 \\
	F(\{|(u_i)_\nu(x)|\}) = 1 & x \in \partial^*\W \cap B_1\\
	F(\{|\nabla u_i^l|\}) = F(\{|\nabla u_i^r|\}) \geq 1 & x \in \partial^{c}\W \cap B_1\\
\end{cases}
\end{equation}
Here $\partial^*\W$ represents the reduced boundary of $\W$, while $\partial^{c}\W$ is the set of ``cusp points'' in $\partial \W$ with blow-up limits $\R^n \setminus H$ for a hyperplane $H$, i.e. points where $\partial \W$ has tangents but with $\W$ on either side of those tangents; $\nabla u_i^r, \nabla u_i^l$ represent the derivative from either side.

In typical examples, the  function $F:\R^N \to \R$ in \eqref{eqn: FBP} will be strictly increasing in each parameter and smooth, however the specific form is not important here. 
The only well-studied example is the variational one, $F(\xi) = \sum_i \mu_i \xi_i^2$ with $\mu_i > 0$, which permits viewing this problem as the stationarity condition for a certain energy. 
%The only well-studied example is $F(\xi) = \sum_i \mu_i \xi_i^2$ with $\mu_i > 0$ in the context of eigenvalue optimization problems. Here the variational structure of $F$ permits this problem to be viewed as the stationarity condition for a certain energy, allowing for the use of more general tools like the Weiss monotonicity formula that are not available for more general $F$.
 We are deliberately vague about the sense in which the last two conditions of \eqref{eqn: FBP} hold, as the point of the method here is \emph{that they are not used directly}.

For $\Omega$ and $u_1, \dots, u_N$ satisfying \eqref{eqn: FBP} (extended by zero on $B_1\setminus\Omega$), let $v$ be any rational linear combination of the $u_1,\dots , u_N$, and apply the ACF formula to $v_+$ and $v_-$; set
\begin{equation}
	\label{eqn: G}
	G = \bigcup_{v} \{x \in \partial \W \cap B_{1/2} : J_x(v_+, v_-, 0^+) > 0 \}.
\end{equation}
As a consequence of Theorems~\ref{thm: main rectifiability} and \ref{thm: uniqueness of blowups}, we have the following.
\begin{theorem}\label{thm: fbp}
	Consider a domain $\Omega$ and a collection of functions $\{ u_i\}_{i=1}^N$ satisfying \eqref{eqn: FBP}. The set $G$ is $\cH^{n-1}$-rectifiable. Moreover, for $\cH^{n-1}$-a.e. $x \in G$ and for each $i=1,\dots , N$, the blow-ups of $u_i$ converge to a unique limiting function of the form 
	$$\alpha_i (x \cdot \nu)_+ + \beta_i (x \cdot \nu)_-$$
	 with $\alpha_i \geq 0$ and $\beta_i \leq 0$.
\end{theorem}
In fact, Theorem~\ref{thm: fbp} holds for functions $\{u_i\}_{i=1}^N$  satisfying only the first two conditions in \eqref{eqn: FBP}.
\begin{proof}
	The $\cH^{n-1}$-rectifiability of $G$ follows directly from Theorem~\ref{thm: main rectifiability} since the countable union of rectifiable sets is rectifiable. For each $x\in G$, there is a $v$ to be a rational linear combination for which $J_x(v_+, v_-, 0^+) > 0$; by Theorem~\ref{thm: uniqueness of blowups}, outside of an $\cH^{n-1}$-null set in $G$,
	the blow-ups of every such $v$ converge to a unique limiting function (depending on $v$) of the form $\alpha (x \cdot \nu)_+ + \beta (x \cdot \nu)_-$ with $ \alpha > 0$, $\beta < 0$. 
	 For each $i =1,\dots, N$ and  for $\e$ small enough (and rational), $J_x((v + \e u_i)_+, (v + \e u_i)_-, 0^+) > 0$. This implies that the blow-ups of each $u_i$ are also unique functions of the same form, except with $\alpha, \beta \in \R$ possibly $0$.
	\end{proof}

Let us explain the relevance of the set $G$ and of Theorem~\ref{thm: fbp}. For a solution of \eqref{eqn: FBP}, let $E$ be the set of points in $\partial \Omega$  where $\Omega$ has Lebesgue density strictly less than one, and let $D = \partial \Omega \setminus (E\cup G)$. In this way, we partition the free boundary $\partial \Omega$ into disjoint sets $D, G$, and $E$. On $E$, one can use blow-up and perturbative methods to classify points as in the scalar case: $E$  is composed of the reduced boundary, which is smooth and relatively open, and a lower-dimensional set of singular points. While this problem is unstudied in the generality presented here, \cite{KL1, KL2,MTV17,CSY} deal with particular cases and could be adapted further.
Next,  Theorem~\ref{thm: fbp} tells us that $G$ is $\cH^{n-1}$-rectifiable. For  a sufficiently stable solution (as described more precisely below), the set $D$ is expected to be empty,  so in particular the entirety of $\partial \Omega = E\cup G$ is $\cH^{n-1}$-rectifiable. On the other hand, for a solution where $D$ is nonempty, then at each $x \in D$ the solution must be degenerate in one of the ways described below.

For a reasonable choice of $F$, one can expect to show that for a solution of \eqref{eqn: FBP}, each $u_i$ is Lipschitz continuous. This property is generally straightforward to establish for even the weakest possible interpretation of this problem (see \cite{CaffSalsa}, for instance). 
Sufficiently ``stable'' solutions are expected to satisfy 
the nondegeneracy condition
\begin{equation} \label{e:nondegenerate}
	\max_{i=1,\dots, N}\sup_{B_r(x)} |u_i| \geq c r \qquad \forall x \in \partial \W, r < r_0;
\end{equation}
for instance, this is expected for Perron solutions, or for minimizing solutions in a variational setting.
Finally, the following secondary nondegeneracy property is not limited by stability but is expected to be true for any sufficiently well-behaved solution: for each $i = 1,\dots, N$, $x \in \partial \Omega$, and $r <r_0$,  
\begin{equation} \label{e:notmodxn}
	\sup_{B_r(x)} u_i^\pm \leq C\left[ \sup_{B_{2r}(x)} u_i^\mp + r \sqrt{\frac{|B_{2r}(x)\setminus \Omega|}{|B_{2r}|}}\right],
\end{equation}
i.e. at points $x$ with Lebesgue density close to $1$, $u^\pm$ must have comparable size.

%The relevance of the set $G$ and of Theorem~\ref{thm: fbp} is that it gives a way to establish $\cH^{n-1}$-rectifiability and uniqueness of blowups for precisely the set of (two phase? cusp? what's the right language?) free boundary points for which ....other methods fail. 
\begin{theorem}
Consider a domain $\Omega$ and a collection of functions $\{ u_i\}_{i=1}^N$ satisfying \eqref{eqn: FBP}. 
	If \eqref{e:nondegenerate} and \eqref{e:notmodxn} hold, then $D$ is empty.
\end{theorem}

%Up to switching the signs of some of the $u_i$, we must have that
%\[
%	\lim_{r \searrow 0}\frac{1}{r}\sup_{B_r(x), i} u_i^- = 0.
%\]
%In other words, near $x$, each $u_1, \dots, u_N$ is asymptotically ``one-phase'' and they all have the same sign to first order. If \eqref{e:nondegenerate} and \eqref{e:notmodxn} hold, this implies that the Lebesgue density of $\W$ at $x$ is strictly less than $1$. 
The set $G$ is not well understood, even in the case of minimizers to the variational form of this problem. In the variational case $F(\{\xi_i\} ) = \sum_i \mu_i \xi_i^2$,  the set $G$ is shown to be regular in the scalar case $N = 1$ \cite{DSV} and in two dimensions \cite{SV}, or rectifiable in arbitrary dimensions \cite{DESV}. Theorem~\ref{thm: fbp} gives an alternative and rather different proof of the main result of \cite{DESV}, and additionally establishes the rectifiability of $G$ in the broader context of non-variational problems.

\medskip

Let us give a second example of how our main theorems above can be applied to analyze free boundaries. Consider a two-phase parabolic equation of Hele-Shaw type: a function $u : \W \times (0, T] \rightarrow \R$ solves
\[
	\begin{cases}
		\Delta_x u(x, t) = 0 & u(x, t) \neq 0\\
		\partial_t u(x, t) = |\nabla u_+|^2 - |\nabla u_-|^2 & u(x, t) = 0\\
		u(x, t) = \phi(x) & x \in \partial \W.
	\end{cases}
\]
Two-phase flows of Bernoulli/Stefan type have been studied extensively in the context of heat transfer, porous media, fluid dynamics, and other models (\cite{CaffSalsa}). In general, the results available are either local regularity theorems under strict starting assumptions on the flow to reduce to a perturbative starting point, or global theorems exploiting monotonicity of $u$ or $\{u > 0\}$ in $t$ by imposing special boundary or initial conditions. A major difficulty in comparison to stationary problems is the lack of geometric measure theoretic results about $\{u = 0\}$.

Here Theorem \ref{thm: main rectifiability} can be applied directly to give that the set of points where $J_x(u_+(\cdot, t), u_-(\cdot, t), 0^+) > 0$ is rectifiable (on each time slice, independently). This suggests that at points of this type, one has a reasonable starting configuration for perturbative results from a linearized two-phase configuration. At points where $J_x(u_+(\cdot, t), u_-(\cdot, t), 0^+) = 0$, one instead hopes for perturbative results from a one-phase configuration (which is better understood, \cite{cjk07,cjk09,cgs19}), or some other more degenerate behavior.

	\section{Preliminaries} 
	This section contains preliminary results that will be used in the paper.
Section~\ref{ssec: notation} provides the basic notation and some initial scaling observations. In section~\ref{ssec: NV}, we recall two Reifenberg-type theorems of Naber and Valtorta. 
In section~\ref{ssec: QACF}, we recall our main result of \cite{AKN1} and prove some slight variants of it that will be applied in this paper.
Finally, section~\ref{ssec: cptness} contains an important lemma that allows us to deduce convergence of the ACF monotonicity formulas of $L^2$-convergent sequences of functions. 

	\subsection{Notation and Basics}\label{ssec: notation}
We use the notation $B_r(x)$ to denote a ball of radius $r>0$ centered at $x\in \R^n$ and $B_r$ to denote $B_r(0)$, and let $\omega_{n}$ denote the volume of the $n$-dimensional Euclidean unit ball. 	Given a set $E \subset \R^n$, let $d(x,E) = \inf_{y \in E} d(x,y)$ and $B_r(E) = \cup_{x \in E} B_r(x)$.

We recall that a set $S\subset \R^n$ is  $\mathcal{H}^{n-1}$-rectifiable if there exist countably many Lipschitz maps $f_j: \R^{n-1} \to \R^n$ such that $\mathcal{H}^{n-1} (S \setminus \bigcup_j f_j(\R^{n-1}) ) =0$, and caution the reader that this definition is sometimes called  {\it countable}  $\mathcal{H}^{n-1}$-rectifiability.

It is worth noting how our quantities behave under rescalings.	Suppose $u,v$ are as in \eqref{eqn: ACF setup}. For $r \in (0,4)$, let $u^{x,r}(z) = u(x+ rz)/r$ and $v^{x,r}(x) = v( x+rz)/r$. Then  
	\begin{align*}
		J_x(\rho ; u, v ) = J_0\left( {\rho}/{r}  ; u^{x,r}, v^{x,r} \right), & \qquad \qquad
		\Gamma^*(u , v) \cap B_r(x) = \Gamma^* \left(u^{x,r}, v^{x,r}\right) \cap B_1(0),\\
		\str_\e(u , v)\cap B_r(x) = \str_{\e} (u^{x,r} , v^{x,r}) \cap B_1(0), & \qquad\qquad
		\str_{\e, \rho} (u, v) \cap B_r(x)  = \str_{\e, \rho/r} (u^{x,r}, v^{x,r}) \cap B_1(0)\,.
	\end{align*}

\subsection{Reifenberg-Type Theorems}\label{ssec: NV}
We recall two important theorems of Naber and Valtorta, the discrete Reifenberg theorem and the rectifiable Reifenberg theorem. 
 Given a Borel measure $\mu$, a point $x \in B_1(0)$, and a scale $r \in (0,1)$, the {\it $(n-1)$-dimensional $L^2$ Jones' beta number} $\beta_\mu (x,r)$ with respect to $\mu$ is defined by
\begin{equation}
\label{eqn: beta numbers} 
\beta_\mu (x,r)^2 =  \inf_L \int_{B_r(x)} \frac{d(y, L)^2}{r^2} \, \frac{d\mu(y)}{r^{n-1}}.
\end{equation}
Here the infimum is taken over all affine hyperplanes $L$ in $\R^n$.  This quantity measures, in a scale invariant $L^2$ sense, how far the support of $\mu$ is to being contained in an affine hyperplane.
Following Jones' use of an $L^\infty$ analogue of \eqref{eqn: beta numbers}, the $L^p$ beta numbers were first introduced by David and Semmes in \cite{DavidSemmesBeta} and have since been used, for instance, in \cite{DavidToro, AzzamTolsa} as well as \cite{NVmain}.

We state the following two Reifenberg-type theorems only in the $(n-1)$-dimensional case in which we will apply them. We refer the reader to Theorems 3.4 and 3.3 in \cite{NVmain} respectively for the statements in greater generality,  as well as to the paper \cite{ENV} for further extensions.

\begin{theorem}[Discrete Reifenberg Theorem]\label{thm: discrete reifenberg}
Fix $n\geq 2.$ 
There are dimensional constants $\delta_0 = \delta_0(n)>0$ and $C_0=C_0(n) >0$ such that the following holds.
	Let $\{B_{r_x}(x)\}$ be  a collection of disjoint balls with $x \in B_1(0)$ and $r_x \in (0,1]$, and let $\mu= \sum_x r_x^{n-1} \delta_x$ be the associated $(n-1)$-dimensional packing measure. If 
	\[
	\int_0^{2r} \int_{B_r(x)} \beta_\mu(z,s)^2 \, d \mu(z) \frac{ds}{s} \leq \delta_0\, r^{n-1}
	\]
	for all $x \in B_1(0)$ and all $r \in (0,1]$, then
	\[
 	\sum_x r_x^{n-1} \leq C_0.
	\]
\end{theorem}
We will apply Theorem~\ref{thm: discrete reifenberg} in Section~\ref{sec: l2} to establish packing estimates.
Later, in Section~\ref{sec: rect}  we will apply the rectifiable Reifenberg theorem to prove that $\Gamma^*_\e$ is $\mathcal{H}^{n-1}$-rectifiable.
\begin{theorem}[Rectifiable Reifenberg Theorem]\label{thm: rect reif}
 Fix $n\geq 2$. There are constants $\delta_0=\delta_0(n)$ and $C_0 =C_0(n)$ such that the following holds. Let $S\subset \mathbb{R}^n$ be a set and let $\mu_S = \mathcal{H}^{n-1}\mres S.$ If
	\[
	\int_0^{2r} \int_{B_r(x)} \beta_\mu(z,s)^2 \, d \mu(z) \frac{ds}{s} \leq \delta_0\, r^{n-1}
	\]
	for all $x \in B_1(0)$ and all $r \in (0,1]$, then $S \cap B_1(0)$ is $\mathcal{H}^{n-1}$-rectifiable, and for each $x \in B_1(0)$ and $0 < r<1$, 
	\[
	\mu_S(B_r(x))\leq C_0 r^{n-1}.
	\]
	
\end{theorem}

\subsection{Quantitative remainder for the ACF monotonicity formula}
\label{ssec: QACF}

In \cite{AKN1} we proved a sharp quantitative version of the ACF monotonicity formula, which will play an important role here. 
The estimate will be used at various different points throughout the paper; the optimal exponent in our quantitative estimate is crucial in the $L^2$ subspace approximation theorem, Theorem~\ref{t:l2est2}.
\begin{theorem} \label{t:ACF} Fix $n\geq2$. There is a dimensional constant $C>0$ such that the following holds.
Let $u, v :B_2(0)\to \R$ be nonnegative continuous functions that satisfy \eqref{eqn: ACF setup}. For any $\rho\in [0,1/2]$,
there exist $\aone, \atwo >0$ and a direction $\nu \in \mathbb{S}^{n-1}$ such that
   \[
  \int_{B_1 \setminus B_{\rho}} \left(u -\aone (x\cdot \nu)^+ \right)^2 + \left(v -\atwo (x\cdot \nu)^-\right)^2
  \leq C\log\left(\frac{J_0(1)}{J_0(\rho)}\right)  \int_{B_1} \left( u^2 +v^2\right)  .  
 \] 
 There is a dimensional constant $\kappa_0$ such that if $\log(J(1)/J(0^+))<\kappa_0^2,$ then $\aone, \atwo$, and $\nu$ can be chosen independently of $\rho.$
\end{theorem}

In \cite{AKN1}, Theorem~\ref{t:ACF} is stated with the additional assumption that $u(0)=v(0)=0$, i.e. that the point where the  ACF monotonicity formula is centered lies in the mutual zero set $\{x: u(x) = v(x) = 0\}$. This assumption can be removed with the following observation and notational caveat. Suppose $u,v$ are as in the statement of Theorem~\ref{t:ACF} and $v(0)>0$. By continuity, there is a radius $r_0$ such that $v(x)>0$, and thus $u(x) =0$, for all $x \in B_{r_0}(0)$. In turn, $J_0(r) =0$ for all $r <r_0$. Let us adopt the convention that, if $J_0(r) =0$, then $\log(J_0(R)/J_0(r)) = +\infty$ for any $R \geq r$. With this convention in place,  Theorem~\ref{t:l2est2} holds and the proof goes through without modification without the assumption that $u(0)=v(0)= 0$.

%
%
%
%
%For the purposes of this paper we will use the following slight improvement of Theorem~\ref{t:ACF}
%\begin{theorem} \label{t:ACFi}
% Let $u, v$ satisfy the assumptions of Theorem~\ref{t:ACF}. Assume also that $\log(J(1)/J(\rho))\leq 1$. Then there exists a constant $C$ depending only on $n$ and $\rho$ such that 
% $\int_{B_1} (u^2+v^2) \leq C \int_{B_\rho}(u^2 + v^2)$, and consequently 
%  \[
%  \int_{B_1 \setminus B_{\rho}} \left(u -\aone (x\cdot \nu)^+ \right)^2 + \left(v -\atwo (x\cdot \nu)^-\right)^2
%  \leq C\log\left(\frac{J(1)}{J(\rho)}\right) \int_{B_{3/4}} (u^2 + v^2).  
% \] 
%\end{theorem}
%
%\begin{proof}
% The proof in the quantitative ACF paper actually shows that if $u,v$ satisfy the assumptions of the theorem, then 
% \begin{equation} \label{e:3forbound}
% \int_{B_1} |\nabla u|^2 + |\nabla v|^2 \leq C \int_{B_{3/4}} u^2 + v^2\,.
% \end{equation}
% Suppose by way of contradiction the result is not true. Then there exists two sequences $u_{k}, v_{k}$ satisfying the assumptions of the theorem, but such that 
% \[
%  \int_{B_{3/4}} u_{k}^2 + v_{k}^2 \leq k \int_{B_{1}} u_{k}^2 + v_{k}^2. 
% \]
% By multiplying by a constant we may assume that $\int_{B_{1}} u_{k}^2 + v_{k}^2=1$. 
% Then there exist two limiting functions $u_{\infty}, v_{\infty}$ that satisfy the assumption of the theorem and with 
% \[
%  \int_{B_{1}} u_{\infty}^2 + v_{\infty}^2 =1 \quad \text{ and }  \int_{B_{3/4}} u_{\infty}^2 + v_{\infty}^2 =0. 
% \]
% But then by \eqref{e:3forbound} we have that $u$ and $v$ are constants. This gives a contradiction.
%\end{proof}
%
%
%
%

We will use Theorem~\ref{t:ACF} in the form above as well as in the form of the following corollary:
\begin{corollary}\label{cor: coeff 1} 
	Fix $n\geq 2$. 
	There are dimensional constants $\rho_0(n)$, $\kappa_0(n)$, and $C=C(n)$ such that the following holds. 
	Let $u, v :B_8(0)\to \R$ be nonnegative continuous functions that satisfy \eqref{eqn: ACF setup}  and suppose $\log(J(1)/J(\rho))\leq \kappa_0^2$. Suppose further that  $\aone = \atwo$ in the conclusion of  Theorem~\ref{t:ACF}. 
	Then for any $\rho < \rho_0$, 
	\begin{equation}
		\label{eqn: coeff 1}
	\int_{B_1\setminus B_\rho} \left(\frac{u }{\aone}- (x\cdot \nu)^+ \right)^2 + \left(\frac{v}{ \atwo}- (x\cdot \nu)^-\right)^2 \leq C\log\left(\frac{J(1)}{J(\rho)}\right)\,.
	\end{equation}
\end{corollary}
\begin{proof}
	Without loss of generality, suppose $\| u\|_{L^2(B_1)} \geq  \| v\|_{L^2(B_1)}$.  By Theorem~\ref{t:ACF} and the triangle inequality,
		\begin{equation}
		\label{eqn: l2 norms a}
\Big| \| u\|_{L^2( B_1 \setminus B_\rho) } - \| \aone\, (x\cdot \nu)^+\|_{L^2( B_1 \setminus B_\rho) } \Big| \leq C\kappa_0\| u\|_{L^2(B_1)}
	\end{equation}
	for a dimensional constant $C$. 
	Since $u$ is nonnegative and subharmonic, $u^2$ is subharmonic and by the mean value property, $\int_{B_\rho} u^2\leq \rho^n \int_{B_1} u^2$. As such, 
	 $\left| \| u\|_{L^2(B_1\setminus B_\rho)}-\| u\|_{L^2(B_1)}\right| \leq \rho^{n/2} \| u\|_{L^2(B_1)}$. The analogous estimate holds for $\aone\,(x\cdot \nu)^+$ as well (since it is subharmonic, or by direct calculation). So, \eqref{eqn: l2 norms a} implies that 
	\begin{equation*}
\Big| \| u\|_{L^2( B_1) } - \| \aone\, (x\cdot \nu)^+\|_{L^2( B_1) } \Big| \leq C\left(\kappa_0 + \rho^{n/2}\right)\| u\|_{L^2(B_1)}\,.
	\end{equation*}
Provided $\rho_0$ and $\kappa_0$ are chosen to be small enough, the coefficient on the right-hand side is bounded above by $1/2$. 		 
So, dividing through by $\aone$ in Theorem~\ref{t:ACF}, we have 
\[
	\int_{B_1\setminus B_\rho} \left(\frac{u }{\aone}- (x\cdot \nu)^+ \right)^2 \leq C\| (x\cdot \nu)^+ \|_{L^{2}(B_{1})}^2 \,\log\left(\frac{J(1)}{J(\rho)}\right) =C \,\log\left(\frac{J(1)}{J(\rho)}\right)  \,	
\]
for a dimensional constant $C$. The same argument can now be repeated with $v$ in place of $u$ and with $\| \aone(x\cdot \nu)^+\|_{L^2(B_1 )}$ on the right-hand side of \eqref{eqn: l2 norms a}. 
\end{proof}

\subsection{Continuity properties for the ACF monotonicity formula}\label{ssec: cptness}
The Alt-Caffarelli-Friedman monotonicity formula is continuous in $x$ for fixed $r>0$, and it is upper semicontinuous in the sense that for any sequence $x_i \to x_0$ and $r_i \to 0^+$, we have 
\begin{equation}\label{eqn: USC}
\limsup_{i \to \infty} J_{x_i}(r_i) \leq J_{x_0}(0^+)\,.
\end{equation}
We leave the proofs of these facts to the reader. 
A less immediate continuity property of the ACF monotonicity formula is given in Corollary~\ref{cor: J convergence} below: if a pair of functions $(u, v)$ as in \eqref{eqn: ACF setup} converges in  $L^2(B_R)$ to a pair of complementary truncated linear functions, then their ACF formulas converge as well up to scale $R$. This key fact will follow from the next lemma.
\begin{lemma}\label{lem: improved convergence} Fix $R>0$. Let $\{u_j\} \subset  W^{1,2}(B_R)$ be a sequence of nonnegative subharmonic functions and let $u\in W^{1,2}(B_R)$ satisfy $u \, \Delta u=0$, i.e. $u$ is harmonic where it is positive. 
	If $ u_j \rightharpoonup u$ in $W^{1,2}(B_R)$ and $u_j \to u$ in $L^2(B_R)$, then $u_j \to u$ in $W^{1,2}(B_R)$ and
	\[
	\lim_{j \to \infty} \int_{B_R} \frac{|\na u_j|^2}{|x|^{n-2}} = \int_{B_R} \frac{|\na u|^2}{|x|^{n-2}}\,.
\]
\end{lemma}
\begin{proof}
	Fix a nonnegative function $\phi \in C^1_c(B_R(0))$. Because $u_j$ is subharmonic and $u_j \phi \geq 0$, we have
	\begin{align*}
		\int_{B_R} |\na u_j|^2 \phi  &  = -\int_{B_R} u_j \, \na u_j \cdot \na \phi + \int _{B_R} \na u_j \cdot \na ( u_j \phi ) \leq -\int_{B_R}  u_j \na u_j \cdot \na \phi.
	\end{align*}
Now, since $u_j \na \phi \to u \na \phi $ strongly in $L^2(B_R)$ and $\na u_j \rightharpoonup \na u$ weakly in $L^2(B_R)$, it follows that 
\begin{align*}
	\limsup_{j \to \infty}  	\int_{B_R} |\na u_j|^2 \phi & \leq  -\lim_{j \to \infty}  \int_{B_R}  u_j\, \na u_j \cdot \na \phi = -\int_{B_R}  u \, \na u \cdot \na \phi = \int_{B_R} |\na u|^2 \phi\,.
\end{align*}
Here the final identity comes from integration by parts and the fact that  $u \, \Delta u =0$. On the other hand, by lower semicontinuity of the energy with respect to $W^{1,2}$ weak convergence, 
\begin{align*}
	\liminf_{j \to \infty}  	\int_{B_R} |\na u_j|^2 \phi  \geq  \int_{B_R} |\na u|^2 \phi\,.
\end{align*}
These two inequalities together tell us that, for any nonnegative $\phi \in C^1_c(B_R(0))$, 
\[
\lim_{j\to \infty} \int_{B_R} |\na u_j |^2 \phi = \int_{B_R} |\na u |^2 \phi\,.
\]
 The first conclusion of the lemma follows from taking a sequence of functions $\phi_k$ that approximate the characteristic function of $B_R$, while the second conclusion follows from taking $\phi_k = |x|^{n-2} \psi_k$ where $\psi_k\in C^1_c(B_R(0))$ vanishes in $B_{1/k}(0)$ and approximates the characteristic function of $B_R$ in $L^2(B_R)$.
\end{proof}
As we mentioned above, Lemma~\ref{lem: improved convergence} has the following useful corollary that will allow us to upgrade $L^2$ convergence of sequences to convergence of their ACF formulas.
\begin{corollary}\label{cor: J convergence}
Let $u_j, v_j :B_2(0)\to \R$ be sequences of nonnegative continuous functions that satisfy \eqref{eqn: ACF setup} with  $\sup_{j} J_0(1; u_{j}, v_{j}) <+\infty$ and suppose $u_{j} \to  \aone (x\cdot \nu)^+$ and  $v_{j} \to \atwo(x\cdot \nu)^-$ in $L^2(B_1(0))$ for some $\aone, \atwo>0$ and $\nu \in \mathbb{S}^{n-1}$. Then, setting $c_* = n^2 \omega_n^2/16$, we have 
	\[
	\lim_{j \to \infty} J_0(r; u_{j}, v_{j} ) = c_*\, \aone^2 \, \atwo^2 \qquad \text{ for all } r \in (0,1]\,.
	\]
\end{corollary}
The constant $c_*$ is such that  $J_0(r;  \aone (x\cdot \nu)^+, \atwo(x\cdot \nu)^-) = c_*\, \aone^2\,\atwo^2$ for any $r>0$.
\begin{proof}
The assumption $\sup_{j} J_0(1; \, u_{j},\, v_{j}) <+\infty$ implies that after passing to a subsequence, either
\begin{equation}\label{eqn: bounded energy}
	\sup_j \int_{B_1} \frac{|\na u_{j}|^2}{|x|^{n-2}} \,dx  < +\infty \qquad  \text{ or } \qquad 
		\sup_j \int_{B_1} \frac{|\na v_{j}|^2}{|x|^{n-2}}, \,dx  < +\infty  
\end{equation}
or both; we assume without loss of generality it holds for $u_j$.  
The weight $|x|^{2-n}$ is bounded below by $1$ on $B_1(0)$, so $\sup_j \int_{B_1(0)} |\na u_{j}|^2\,dx <+ \infty$ and thus up to a further subsequence, $u_{j} \rightharpoonup \aone (x\cdot \nu)^+$ in $W^{1,2}(B_1(0))$.
 The truncated linear function $\aone (x\cdot \nu)^+$ is harmonic where it is positive, so Lemma~\ref{lem: improved convergence} implies that
\[
\lim_{j \to \infty} \int_{B_1(0)} \frac{|\na u_{j}|^2}{|x|^{n-2}} \,dx = \int_{B_1(0)} \frac{|\na \, \aone (x\cdot \nu)^+|^2}{|x|^{n-2}} \,dx  = \aone^2 \sqrt{c_*}\,.
\]
Here $c_*$ is  the dimensional constant defined in the corollary statement and the final equality is an elementary calculation.
% carried out in Lemma~\ref{lem: simple comp} below. 
Now, since $\aone^2 \sqrt{c_*}>0$, the assumption $\sup_{j} J_0(1; \, u_{j},\, v_{j}) <+\infty$ now tells us that \eqref{eqn: bounded energy} holds for $v$ along the same subsequence. Repeating the same argument shows that 
\[
\lim_{j \to \infty} \int_{B_1(0)} \frac{|\na v_{j}|^2}{|x|^{n-2}} \,dx = \atwo^2 \sqrt{c_*}
\]
and we reach the conclusion of the corollary along a subsequence. Any subsequence has a further subsequence for which the conclusion of the corollary holds, so it  holds for the full sequence.
\end{proof}

\section{The $L^2$ Subspace Approximation and Packing Estimates}\label{sec: l2}
This section has two main goals. First, we prove an estimate known as the $L^2$ subspace approximation in Theorem~\ref{t:l2est2} below. This estimate quantitatively relates the Jones' $L^2$ beta numbers and the drop in the monotonicity formula from one scale to the next, and plays a key role in the Naber-Valtorta framework. The statement of Theorem~\ref{t:l2est2} is analogous to $L^2$ subspace approximations in other contexts.
%(e.g. \cite[Theorem 7.1]{NVmain},\cite[Theorem 45]{NVApprox}, \cite[Theorem 5.1]{EE19}). 
However, our proof is different from the standard one and in particular circumvents the use of the eigenvalues and eigenvectors of the ``inertia matrix'' associated to a measure.
The proof (as well as an adaptation of the standard proof) crucially relies on the sharp quantitative remainder term in for the ACF monotonicity formula of \cite{AKN1}. 

\begin{theorem}[$L^2$ Subspace Approximation] \label{t:l2est2} 
Fix $n\geq 2.$ There exist positive dimensional constants  $\kappa$  and $C$ such that the following holds. 
Let $u, v :B_{10}(0)\to \R$ be nonnegative continuous functions satisfying \eqref{eqn: ACF setup}.
%
%\footnote{Annoying: I guess we need $B_9(0)$ at least or else we can assume $r \in (0,1/2)$, maybe the latter is more elegant.}
%
 Let $\mu$ be a finite Borel measure. 
 %with $\text{spt } \mu \subset \Gamma$.
 For any $r \in(0,1)$ and $x \in B_1(0)$ such that 
 $\log(J_x(8r)/J_x(r)) < \kappa$, we have 
 \begin{equation} \label{e:intest}
 \begin{aligned}
  \beta_{\mu}(x,r)^2 &\leq \frac{C}{r^{n-1}} \int_{B_r(x)}  \log\left(\frac{J_{y}(8r)}{J_{y}(r)}\right) d \mu(y). 
  \end{aligned}
 \end{equation}
\end{theorem}
 Here $ \beta_{\mu}(x,r)^2$ is the $(n-1)$-dimensional $L^2$ Jones' beta number defined in \eqref{eqn: beta numbers}.
Following \cite{NVmain}, Theorem~\ref{t:l2est2}  will be applied twice: when $\mu$ is the packing measure associated to a cover to prove the Proposition~\ref{lem: packing estimate for disjoint balls with small energy drop} below, and  with $\mu = \cH^{n-1} \mres \str_\e$ in section~\ref{sec: rect} in the proof of  Theorem~\ref{thm: main rectifiability}.

The second main goal of this section is to prove the packing estimates of  the following proposition. Roughly speaking, this proposition says that if a $\str_{\e, r}$ has a Vitali cover by balls with small drop in the ACF monotonicity formula at the centers, then the cover has a uniform $(n-1)$-dimensional packing bound.
The statement and proof of Proposition~\ref{lem: packing estimate for disjoint balls with small energy drop} are a standard part of the Naber-Valtorta framework.
\begin{proposition}[Packing Estimate]
	\label{lem: packing estimate for disjoint balls with small energy drop}
Fix $n\geq 2$ and $0<\e< 1/2$. There exist $\etaone=\etaone(n,\e)>0$ and $C(n)$ such that the following holds. Let $u, v :B_{10}(0)\to \R$ be nonnegative continuous functions satisfying \eqref{eqn: ACF setup}  and let $\bar J = \sup_{y \in B_1(0)} J_y (4)$. 
 If $\{B_{2r_p}(p)\}_p$ is a collection of disjoint balls with $p \in B_1(0)$ satisfying
 \begin{equation} \label{e:assume}
 J_p(\etaone\, r_p) \geq \bar J - \etaone, 
 \qquad
  p \in \str_{\epsilon,  R} \cap B_1(0), 
 \qquad R \leq r_p \leq 1, 
 \end{equation}
 then 
 \begin{equation}
 	\label{eqn: packing conclusion}
 	  \sum_{p} r_p^{n-1} \leq C(n). 
 \end{equation}
\end{proposition}

\subsection{Two initial lemmas} We prove two initial lemmas in preparation for the proof of Theorem~\ref{t:l2est2}. The first one will be applied when $\eta$ is small, and says that if two pairs of complementary truncated linear functions are close in an $L^2$ sense on an annulus, then their hyperplane interfaces are also close in a quantitative way.
\begin{lemma}  \label{l:qconesplitting} 
There exists a positive dimensional constant $C$ such that the following holds.
 Let $y_1,y_2 \in B_1(0)$ and $\nu_1, \nu_2 \in \mathbb{S}^{n-1}$.  
 Let 
 \begin{align*}
 	\ell_{1}^+ = \aone_{1} [(x-y_1)\cdot \nu_1]^+ , & \qquad  \ell_{1}^- =\atwo_{1} [(x-y_1) \cdot \nu_1]^-, \\ 
\ell_{2}^+ = \aone_{2} [(x-y_2)\cdot \nu_2]^+ , & \qquad \ell_{2}^- =\atwo_{2} [(x-y_2) \cdot \nu_2]^-.
  \end{align*}
   Assume that $\max\{\aone_{1},\atwo_{1}\} \geq c$. For any $\eta \leq 1/16$, if
 \[
  \int_{B_6(0) \setminus B_2(0)} (\ell_{1}^+ - \ell_{2}^+)^2 + (\ell_{1}^- - \ell_{2}^-)^2 \leq \eta, 
 \]
 then  $d(y_2, L)\leq \frac{C}{c} \sqrt{\eta}$, where $L =\{ x\in \R^n: (x-y_1)\cdot \nu_1=0\}$ is the mutual zero set of $\ell_{1}^+ $ and $\ell_{1}^-$.
\end{lemma}

\begin{proof}
Let $\sqrt{\eta}= d(y_2,  L)$. 
 We will prove the lemma by showing that 
 \begin{equation}
 	\label{eqn: integral lower bound}
  \int_{B_6(0) \setminus B_2(0)} (\ell_{1}^+ - \ell_{2}^+)^2 + (\ell_{1}^- - \ell_{2}^-)^2 \geq c^2 \eta.  
 \end{equation}
 Up to a rotation we may assume that $\nu_1 = e_n$, $\aone_1 \geq c$, and $\nu_2$ is a linear combination of $e_1$ and $e_n$. 
Since $y_1, y_2 \in B_1(0)$, it suffices to replace the domain of integration in \eqref{eqn: integral lower bound} by the smaller set $B_5(y_1) \setminus B_3(y_1) \subset B_6(0) \setminus B_2(0)$; up to a translation, we may take $y_1$ to be the origin and $y_2 \in B_2(0)$.
 %, and thus integrate over $B_{5}(y_1)\setminus B_3(y_1)$. 
 After these normalizations, $L = \{x \in \R^n : x \cdot e_n  = 0\}$.
 Since we are only finding $d(y_2, L)\leq C \sqrt{\eta}$, we may assume that the $i$-th coordinates of $y_2$ satisfy $y^i_2=0$ for $1 \leq i <n$. 
 For simplicity we will treat the case when $y^n_2 \geq 0$; the case in which $y_2^n <0$ is similar. By relabeling $a_2$ and $b_2$ if necessary, we may assume $\nu_2 \cdot e_n \geq 0$. Finally, by symmetry we will assume  $\nu_2 \cdot e_1 \geq 0$.
With these  assumptions we have the following simplifications:
 \begin{align*}
 	\ell_1^+ = \aone_{1} x^n_+,  & \qquad  \ell_1^- = \atwo_{1} x^n_- , \\ 	\ell_2^+ = \aone_{2} [(x- \sqrt{\eta}e_n) \cdot (\gamma_1 e_1 + \gamma_2 e_n)]^+ , & \qquad \ell_2^- = \atwo_{2} [(x-\sqrt{\eta}e_n)\cdot (\gamma_1 e_1 + \gamma_2 e_n)]^-. 
 \end{align*}
 So, we may find  constants $c_1, c_2\geq 0$ such that  $\ell_2^+$ can be expressed as 
\[
\ell_2^+ = \aone_{2}[(x- \sqrt{\eta}e_n) \cdot (\gamma_1 e_1 + \gamma_2 e_n)]^+ = [c_1\aone_{1}(x^n - \sqrt{\eta}) + c_2 \aone_{1} \sqrt{\eta} x^1]^+.
\]
It will also suffice to integrate over the set 
$$\tilde{A}=(B_5(y_1) \setminus B_3(y_1)) \cap \{x^n >\sqrt{\eta}\} \cap \{ x^1 >0\}. $$
We note on $\tilde{A}$ that $\ell_1^-= \ell_2^- = 0$ . 
%a_1 x^n$ and $\mathbf{l}_2=c_1\aone_{1}(x^n - \sqrt{\eta}) + c_2 \aone_{1} \sqrt{\eta} x^1$.
  We will show that $|\ell_1^+ - \ell_2^+|\geq \aone_1 \sqrt{\eta}/8$ pointwise on a subset of $\tilde{A}$ with measure universally bounded below. 
 To do this, it will suffice to have either 
 \[
  \aone_{1} x^n - c_1 \aone_{1} (x^n - \sqrt{\eta}) - c_2 \aone_{1} \sqrt{\eta} x^1 \geq \aone_{1} \sqrt{\eta}/8,
 \]
 or 
 \[
  \aone_{1} x^n - c_1 \aone_{1} (x^n - \sqrt{\eta}) - c_2 \aone_{1} \sqrt{\eta} x^1 \leq - \aone_{1} \sqrt{\eta}/8,
 \]
 on a sufficiently large subset of $\tilde{A}$. 
 The above inequalities simplify to 
 \begin{equation} \label{e:simp1}
  \frac{1-c_1}{\sqrt{\eta}} x^n \geq \frac{1}{8} - c_1 + c_2 x^1,
 \end{equation}
 and 
  \begin{equation} \label{e:simp2}
  \frac{1-c_1}{\sqrt{\eta}} x^n \leq \frac{-1}{8} - c_1 + c_2 x^1. 
 \end{equation}
We again note that $x^n \geq \sqrt{\eta}$ and $x^1 \geq 0$. We break up the proof into several cases. 

\textbf{Case 1a}: $0 \leq c_1 \leq 1$ and $c_2 \leq 3/8$. If $0 \leq x^1 \leq 1$ (and since $x^n/\sqrt{\eta}\geq 1$), then \eqref{e:simp1} holds. 

\textbf{Case 1b}: $0 \leq c_1 \leq 1$ and $c_2 \geq 3/8$ and $1-c_1 \geq c_2 \sqrt{\eta}$. If $0\leq x^1 \leq 1$ and $x^n \geq 2$, then \eqref{e:simp1} holds. 

\textbf{Case 1c}: $0 \leq c_1 \leq 1$ and $c_2 \geq 3/8$ and $1-c_1 \leq c_2 \sqrt{\eta}$. If $ x^1 \geq 4$ and $x^n \leq 1$, then \eqref{e:simp2} holds. 

\textbf{Case 2a}: $c_1 >1$ and $c_2 \geq 3/8$. If $x_1 \geq 3$ (and since $x^n/\sqrt{\eta}\geq 1$), then \eqref{e:simp2} holds. 

\textbf{Case 2b}: $c_1 >1$ and $c_2 \leq 3/8$ and $1-c_1 \geq -c_1 \sqrt{\eta}$. If $x_n \leq 1/2$ and $0 \leq x^1 \leq 1$, then \eqref{e:simp1} holds. 

\textbf{Case 2c}: $c_1 >1$ and $c_2 \leq 3/8$ and $1-c_1 \leq -c_1 \sqrt{\eta}$. If $x_n \geq 2$, then \eqref{e:simp2} holds.

%\textbf{Case 1}: $c_1 >1$. By choosing $\sqrt{\eta} \leq 1/2$, then \eqref{e:simp2} holds for $x^n \geq 1$ which is a substantial subset of $\tilde{A}$. 

%\textbf{Case 2}: $0 \leq c_1 \leq 1$ and $c_2 \geq 1/4$. Consider the set $\tilde{A} \cap \{|(x^1, \ldots, x^{n-1},0)|<1\}$. If \eqref{e:simp1} does not hold on this set, then $x^n$ will 
%only decrease when $\tilde{A}\cap \{2\leq x^1\leq 3\}$. Furthermore, since $c_2 \geq 1/4$, we have that \eqref{e:simp2} will hold on $\tilde{A}\cap \{2\leq x^1\leq 3\}$.

%\textbf{Case 3}: $0 \leq c_1 \leq 1/2$ and $c_2 \leq 1/4$. Let $ 0 \leq x^1 \leq 1$. Then for $\sqrt{\eta}\leq 1/2$, inequality \eqref{e:simp1} will hold for $x^n \geq 2$. 

%\textbf{Case 4}: $1/2 \leq c_1 \leq 1$ and $c_2 \leq 1/4$. Choose $0 \leq x^1 \leq 1/4$ and \eqref{e:simp1} holds since the right hand side is negative. 

In each of the above exhaustive cases, we obtain a subset of $\tilde{A}$ with positive measure bounded below so that either \eqref{e:simp1} or \eqref{e:simp2} holds. Thus, the result is proven. 
 \end{proof}

The next lemma shows that if an admissible pair $u,v$ are normalized on $B_8(0)$, then the best-approximating truncated linear functions chosen with respect to any $z \in B_1(0)$ will be nondegenerate.
\begin{lemma} \label{l:transfer}
 Let $u,v$ be an admissible pair on $B_9(0)$ with $\|u+v\|_{L^2(B_8)}=1$. Then there exist positive dimensional constants $c_0$ and $\kappa$ such that if   $\log(J_0(8)/J_0(1))\leq \kappa$, and 
 if $z \in B_1$ with $\log(J_z(8)/J_z(1))\leq \kappa$ and 
 \begin{equation} \label{e:plane}
 \int_{B_{7}(z) \setminus B_{1}(z)} [u(y) - a_z ((y-x) \cdot \nu_z)^+]^2 + [v(y) - b_z ((y-x) \cdot \nu_z)^-]^2 \ dy < \kappa, 
 \end{equation}
 then $a_z + b_z \geq c$. 
\end{lemma}

\begin{proof}
 Since $u,v$ are both subharmonic (and consequently $u^2,v^2$ are also subharmonic), we have that 
 \[
  \int_{B_8(0) \setminus B_1(0)} u^2  + v^2 \geq \frac{7}{8}. 
 \] 
 From Theorem \ref{t:ACF} we have the existence of $a_0,b_0>0$ and a direction $\nu_0 \in \mathbb{S}^{n-1}$ satisfying the conclusion of Theorem \ref{t:ACF} at the origin. It then follows that for $\kappa$
 chosen small enough we have 
 $\| a_0 (x \cdot \nu)^+  + b_0 (x \cdot \nu)^-\|_{L^2(B_8(0) \setminus B_1(0))} \geq 3/4$. From Theorem \ref{t:ACF} we also have the existence of $a_z,b_z$ and $\nu_z$
 satisfying \eqref{e:plane}. From the structure of affine functions we have that 
 \[
  \int_{B_5(0) \setminus B_3(0)} a_0 (x \cdot \nu)^+  + b_0 (x \cdot \nu)^- \geq c(n) ,
 \]
 and 
 \[
  \int_{B_5(0) \setminus B_3(0)} a_z (x \cdot \nu-z)^+  + b_0 (x \cdot \nu-z)^- \geq c(n) \int_{B_7(z) \setminus B_1(z)} a_z (x \cdot \nu-z)^+  + b_0 (x \cdot \nu-z)^-.
 \]
 Using now the triangle inequality, we conclude that 
 \[
   \int_{B_5(0) \setminus B_3(0)} a_z (x \cdot \nu-z)^+  + b_0 (x \cdot \nu-z)^- \geq c(n)
 \]
 for a new dimensional constant $c(n)$. The conclusion then follows for a new dimensional constant $c_0$. 
 \end{proof}

\subsection{Proof of the $L^2$ Subspace Approximation}

With Lemmas ~\ref{l:qconesplitting} and \ref{l:transfer} in hand, we can now prove Theorem~\ref{t:l2est2}. The key ingredient is the quantitative remainder term for the ACF formula, Theorem~\ref{t:ACF}.
\begin{proof}[Proof of Theorem~\ref{t:l2est2}]
By scaling and translation we may assume without loss of generality that $r=1$ and $x=0$.  We may divide by a positive constant leaving the quotient on the right-hand side of \eqref{e:intest} invariant, so we may assume that
 that 
 $\| u+v \|^2_{L^2(B_8)}=1$. 
 %From subharmonicity and the same reasoning as in Lemma \ref{l:transfer}, we have that $\| u + v \|^2_{L^2(B_8 \setminus B_1)} \geq 7/8$.

We select a good competitor hyperplane $L$ in the definition of $\beta_{\mu}(0,1)$ in the following way. Let 
\[
\bar x = \text{argmin} \left\{ \log\left( \frac{J_x(8)}{J_x(1)}\right)  \ : \ x \in \overline{B}_1(0)\right\}.
\]
This exists because $x \mapsto J_x(r) $ is continuous and $\overline{B}_1(0)$ is compact. By assumption $\log(J_{\bar{x}}(8)/J_{\bar{x}}(1))<\kappa$. 
Notice that  $B_7(\bar x) \subset B_8$, and so our normalization implies that $\int_{B_7(\bar x)} u^2 + v^2 \leq 1$. So, applying Theorem~\ref{t:ACF} on $B_7(\bar x)$, we find a  pair of truncated linear functions $\ell^{\pm}$ supported on complementary half-planes such that 
\begin{equation}
	\label{e: Ann bar x}
  \int_{B_7(\bar x) \setminus B_1(\bar x)} 
  \left(u -\ell^+\right) +\left( v - \ell^-\right)^2\,
  \leq C\log\left(\frac{J_{\bar x}(7)}{J_{\bar x}(1)} \right) \int_{B_7(\bar x)}  u^2 + v^2 \leq C\log\left(\frac{J_{\bar x }(8)}{J_{\bar x }(1)} \right). 
\end{equation}
In the final inequality we also used the monotonicity of the ACF formula.  Let 
$
L = \{ \ell^{\pm}  =0\}
$
be the hyperplane interface between the supports of $\ell^{\pm}$.
By the same reasoning, for each $z \in \text{spt} \mu \cap B_1(0)$, we apply  Theorem~\ref{t:ACF} to obtain a pair of truncated linear functions 
$\ell^{\pm}_z(x)$ supported on complementary half-planes
such that 
%$\ell^+_z(x) = \aone_z((x-y_z)\cdot \nu_z)^+$ and $\ell^-_z(x) = \atwo_z((x-y_z)\cdot \nu_z)^-$ with $\aone_z,\atwo_z>0$ and $\nu_z \in \mathbb{S}^{n-1}$ such that
  \begin{equation}
  	\label{e: Ann z}
  \int_{B_7(z) \setminus B_1(z)} 
  %(u_1 - \alpha_{1,i} ((x-y_i) \cdot \nu_i)^+)^2 + (u_2 - \alpha_{2,i} ((x-y_i) \cdot \nu)^-)^2 
  \left(u -\ell^+_z\right) +\left( v - \ell^-_z\right)^2\,
  \leq C\log\left(\frac{J_{z}(7)}{J_{z}(1)} \right) \int_{B_7(z)}  u^2 + v^2 \leq C\log\left(\frac{J_{z}(8)}{J_{z}(1)} \right). 
 \end{equation}
Now, since the domains of integration in \eqref{e: Ann bar x} and \eqref{e: Ann z} both contain the annulus $B_6\setminus B_2$, we use the triangle inequality and the choice of $\bar x$ to deduce that 
 \begin{equation}
 	\label{eqn: difference in linears}
 \begin{aligned}
  \int_{B_6 \setminus B_2} \left|\left( \ell^+ + \ell^-\right)  -\left(  \ell_z^+ +\ell_z^-\right)\right|^2  &\leq 2 \int_{B_6 \setminus B_2} (\ell^+-u)^2 + (\ell^- -v)^2 + (\ell^+_z -u)^2 +(\ell_z^- -v)^2  \\
  & \leq C \left(\log\left(\frac{J_{\bar{x}}(8)}{J_{\bar{x}}(1)} \right)+ \log\left(\frac{J_{z}(8)}{J_{z}(1)} \right) \right) \leq C  \log\left(\frac{J_{z}(8)}{J_{z}(1)} \right)%\int_{B_8} u_1^2 + u_2^2 \\
 \end{aligned}
  \end{equation}
  for any $z \in B_1(0)$.
  Now, let us split the support of $\mu$ into two pieces,  letting
\begin{align*}
	G_\mu& = \{z \in \text{spt } \mu \cap B_1(0) : \log (J_z(8)/J_z(1)) \leq \kappa \} \,, \\
	A_\mu &= \{z \in \text{spt } \mu \cap B_1(0): \log (J_z(8)/J_z(1)) >\kappa\}
\end{align*}
  If $z \in G_\mu$, then from Lemma \ref{l:transfer} and \eqref{e: Ann z}
  we have that $\aone_z+ \atwo_z\geq c_0$, where $\aone_z$ and $\atwo_z$ are the slopes of $\ell_z^+$ and $\ell_z^-$ respectively.
So, we can apply Lemma \ref{l:qconesplitting}; together with \eqref{eqn: difference in linears} this tells us that $d(z, L)^2  \leq C\, {\log(J_{z}(8) / J_{z}(1))}$ for any $z \in G_\mu$. Integrating this inequality over $G_\mu$ with respect to the measure $\mu$, we have 
\begin{equation}
	\label{eqn: gmu estimate}
%	\footnote{We may need the better bound on $B_{3/4}$ because we are controlling the $L^2$ norm on $B_8(0)$ but not $B_8(z)$.}
\int_{G_\mu} d(z, L)^2 \, d \mu(z) \leq C \int_{G_\mu} \log\left(\frac{J_{z}(8) }{ J_{z}(1)}\right) \, d\mu(z).
\end{equation}
On $A_\mu$, the analogous estimate holds for trivial reasons: 
for any $z \in B_1(0)$, $d (z, L) \leq C_n$, and thus $d (z, L) \leq \frac{C_n \kappa}{\kappa} \leq \frac{C_n}{\kappa} {\log(J_{z}(8) / J_{z}(1))}$ for any $z \in A_\mu$. Integrating over $A_\mu$ with respect to $\mu$ tells us that 
\begin{equation}
	\label{eqn: amu estimate}
\int_{A_\mu} d(z, L)^2 \, d \mu(z) \leq C \int_{A_\mu} \log\left(\frac{J_{z}(8) }{ J_{z}(1)}\right) \, d\mu(z).
\end{equation}
The conclusion of the theorem follows by summing up \eqref{eqn: gmu estimate} and \eqref{eqn: amu estimate} and using the definition of $\beta_\mu(0,1)$:
 \[
 \begin{aligned}
  \beta_{\mu}(0,1)^2 &\leq \int_{B_1(0)} d(z, L)^2 d \mu(z) 
  % &=  \int_{G_\mu } d(z, L)^2 d \mu(z) +  \int_{A_\mu } d(z, L)^2 d \mu(z)  
%  &\leq C \int_{G_\mu} \log\left(\frac{J_{z}(8) }{ J_{z}(1)}\right) \, d\mu(z) + C \int_{A_\mu} \log\left(\frac{J_{z}(8) }{ J_{z}(1)}\right) \, d\mu(z)\\
  & \leq C \int_{B_1(0)} \log\left(\frac{J_{z}(8) }{ J_{z}(1)}\right) \, d\mu(z)\,.
   \end{aligned}
 \]
 This completes the proof.
\end{proof}

\subsection{Packing Estimates}
Next, we apply the $L^2$ estimate of Theorem~\ref{t:l2est2} to prove Proposition~\ref{lem: packing estimate for disjoint balls with small energy drop}.
The proof is an adaptation to our setting of a by-now standard  induction argument using the $L^2$ subspace approximation and the Discrete Reifenberg Theorem.
\begin{proof}[Proof of Proposition~\ref{lem: packing estimate for disjoint balls with small energy drop}]
{\it Step 1: } Let $\eta \in (0,1)$ be a fixed number depending on $n$ to be specified later in the proof, and set $\etaone =\eta\, \e/2$. 
 Basic algebra shows that any $p$ and $r_p$ satisfying \eqref{e:assume} will also satisfy
  \begin{equation} \label{e:assume2}
  \log\left(\frac{J_p(2)}{J_p(\eta\, r_p)} \right) \leq \eta, \qquad p \in \str_{\epsilon, R} \cap B_1(0), \qquad R \leq r_p \leq 1, 
 \end{equation}
%
%
% 
%\begin{equation}
%	\label{e: conseq of assume}
%	\log\left(\frac{J_p(2)}{J_p( \eta \, r_p)}  \right) \leq \eta\,.
%\end{equation}
Indeed, our choice guarantees that $\bar J - \etaone \geq 2/\e = \eta/\etaone$, and so from the definitions of $\bar J$ and $\str_{\e, R}$, 
\[
\frac{J_p(2)}{J_p(\eta \, r_p)} \ \leq \ \frac{J_p(2)}{J_p(\etaone \, r_p)}\  \leq\  \frac{\bar J}{   \bar J - \etaone } \ = \ 1+\frac{ \etaone}{\bar J -\etaone} \leq 1 + \eta;
\]
then the first part of \eqref{e:assume2} follows from the concavity of the logarithm.
The second and third parts of \eqref{e:assume2} are the same as \eqref{e:assume} and thus hold by assumption. So, to establish the proposition, it suffices to show there exists $\eta(n)$ such that \eqref{eqn: packing conclusion} holds for any collection of disjoint balls $\{B_{2r_p}(p)\}_p$ satisfying
\eqref{e:assume2}.\\

{\it Step 2:}
 Let  $\kappa>0$ be chosen according to Theorem \ref{t:l2est2}, and let $\eta \leq \kappa$. We let $r_i = 2^{-i}$. For each integer $i \in \mathbb{N}$, define the packing measure
 \begin{equation*} \label{e:packingmu}
  \mu_i = \sum_{r_p \leq r_i} r_p^{n-1} \delta_p\,,
 \end{equation*}
and let $\beta_i(x,r) := \beta_{\mu_i}(x,r)$ denote  the corresponding $(n-1)$-dimensional $L^2$ Jones beta number as defined in \eqref{eqn: beta numbers}. 
In this notation, the conclusion \eqref{eqn: packing conclusion} of the lemma  is $\mu_0(B_1(0))\leq c(n)$. In order to prove this, we will argue inductively to prove that
	\[
	\int_0^{2r} \int_{B_r(x)} \beta_\mu(z,s)^2 \, d \mu(z) \frac{ds}{s} \leq \delta_0\, r^{n-1}
	\]
	for all $x \in B_1(0)$ and all $r \in (0,1]$, at which point we can apply Theorem~\ref{thm: discrete reifenberg}. 
% $\mu_i(B_{r_i}(x)) \leq C(n) r_i^{n-1}$  for each $i\geq 0$ and $x \in B_2(0)$; the case $i=0$ and $x =0$ is the desired conclusion. 
More specifically, we argue by induction to show that, for $r_i \leq 2^{-4}$
\begin{equation} \label{e:induct}
 \sum_{r_j \leq 2r_i} \int_{B_{2r_i}(x)} \beta_i(z,r_j)^2 d \mu_i(z) \leq \delta_{0}\, r_i^{n-1} \quad \text { for all } x \in B_1(0).  \tag{$*_i$}
\end{equation}
By the Discrete Reifenberg Theorem~\ref{thm: discrete reifenberg}, whenever \eqref{e:induct} holds we have
\begin{equation} \label{e:induct2}
 \mu_i(B_{r_i}) \leq C_0 r_i^{n-1} \quad \text{ for all } x \in B_1(0). 
\end{equation}
Here $\delta_0$ and $C_0$ are the dimensional constants from Theorem~\ref{thm: discrete reifenberg}.
Notice that \eqref{e:induct} vacuously holds for $i$ large enough such that $r_i <R$, as in this case $\mu_i \equiv 0$. 
Also, if $x \in \text{spt} \mu_j$ and $j \geq i$, then by disjointness we have
 \begin{equation}  \label{e:ij}
  \beta_i(x,r_j) = 
  \begin{cases}\ 
   \beta_j(x,r_j) \quad &  \text{ if } x \in \text{spt} \mu_j, \\
 \   0 \quad  &\text{ otherwise}. 
  \end{cases}
 \end{equation}
Thanks to \eqref{e:assume2}, we can apply Theorem~\ref{t:l2est2} to $\mu_i$ whenever $r_i < 2^{-4}$, finding that for any $x $ that is a center of a ball in our collection, we have 
\begin{equation} \label{e:l2est1}
 \beta_i(x,r_i)^2 \leq \frac{C}{r_i^{n-1}} \int_{B_{r_i}(x)} \log\left(\frac{J(8r_i,y)}{J(r_i,y)}\right) d \mu_i(y)\,.
\end{equation}
Suppose now that the inductive hypothesis $(*_j)$ holds for all
$j \geq i+1$. Fix $x \in B_1(0)$. We first claim that 
\begin{equation} \label{e:packest}
 \mu_{i-1}(B_{4r_i}(x)) \leq M r_i^{n-1}  \quad \text{ and for all } x \in B_1(0),
\end{equation}
where $M= C(n) C_{0}$ with $C(n)$ a dimensional constant. To prove this claim, note that we have $\mu_{i-1}(B_{4r_i}(x))=\mu_{i+1}(B_{4r_i}(x)) + \sum r_p^{n-1}$ where we sum
over $p \in B_{4r_i}(x)$ having $r_{i+1}<r_p\leq r_{i-1}$. Since the $B_{2 r_p}(p)$ are disjoint we sum over at most $C(n)$ points. Also, there are at most $C(n)$ points in $B_{4r_i}(x)$
that are pairwise $r_{i+1}$ distant from each other, so we can cover $B_{4r_i}(x)$ with $C(n)$ balls with center at points $p$. We then use our induction hypothesis with $i+1$ and \eqref{e:induct2} to conclude that 
\[
\mu_{i+1}(B_{4r_i}(x)) \leq \sum \mu_{i+1}(B_{r_{i+1}}(y)) \leq C(n) C_{0}(n) r_{i+1}^{n-1},
\]
which finishes the claim \eqref{e:packest}.
%By the same simple packing argument as in (??) we can assume 
%\begin{equation} \label{e:packest}
% \mu_j(B_{4r_j}(x)) \leq M r_j^{n-1} \quad \text{ for all } j \geq i-2, \quad \text{ and for all } x \in B_1(0),
%\end{equation}
%where $M=c(n)C_{DR}$. 
Now, for any $j \geq i-1,$ i.e. for $r_j \leq 2 r_i$, we have by \eqref{e:ij}, \eqref{e:l2est1}, and Fubini respectively that 
\begin{align*}
 \int_{B_{2r_i}(x)} \beta_i(z,r_j)^2 d \mu_i(z)= \int_{B_{2r_i}(x)} \beta_j(z,r_j)^2 d \mu_j(z)
 & \leq \frac{C}{r_j^{n-1}}   \int_{B_{2r_i}(x)} \int_{B_{r_j}(z)}  \log\left(\frac{J_y(8r_j)}{J_y(r_j)}\right) d \mu_j(y) d \mu_j(z) \\
 & \leq \frac{C}{r_j^{n-1}}  \int_{B_{2r_i+r_j}(x)} \!\! {\mu_j (B_{r_j}(y))}\log\left(\frac{J_y(8r_j)}{J_y(r_j)}\right)  d \mu_j(y)\,.
\end{align*}
Now, applying \eqref{e:packest} and the ordering of the measures we have $\mu_j (B_{r_j}(y))/r_j^{n-1} \leq M$. So, summing over the expression above over all $r_j \leq 2r_i$ and  recalling \eqref{e:assume2}, we find that
\begin{align*}
	\sum_{r_j \leq 2r_i} \int_{B_{2r_i}} \beta_i(z,r_j)^2 d \mu_i(z) 
	&\leq CM  \sum_{r_j \leq 2 r_i}  \int_{B_{4r_i}(x)}\log\left(\frac{J_y(8r_j)}{J_y(r_j)}\right)\, d\mu_j(y)\\
	&\leq CM   \int_{B_{4r_i}(x)}  \sum_{r_j \leq 2 r_i}\log\left(\frac{J_y(8r_j)}{J_y(r_j)}\right)\, d\mu_{i-1}(y)\\
	&\leq CM   \int_{B_{4r_i}(x)} \log\left(\frac{J_p(2)}{J_p(r_p)}\right)\, d\mu_{i-1}(y)
	\leq CM \eta\, \mu_{i-1}(B_{4r_i}(x)).
%	 &\leq CM\!\!\!\!\sum_{p \in B_{4r_i}(x) \cap \text{spt} \mu_{i-1}}\!\!\!\! r_p^{n-1} \log\left(\frac{J_p(2)}{J_p(r_p)}\right) \,.
\end{align*}
Applying \eqref{e:packest} again and then choosing $\eta$ small enough, this establishes
%\[
% \sum_{r_j \leq 2r_i} \int_{B_{2r_i}(x)} \beta_i(z,r_j)^2 d \mu_i(z) \leq \delta_0 r_i^{n-1}. 
%\]
 the induction step and so \eqref{e:induct} holds for all $i$. Now, recalling \eqref{e:ij}, from \eqref{e:induct} we find that for all $x \in B_1(0)$ and all $i \in \mathbb{N}$,
\[
 \sum_{r_j \leq 2r_i} \int_{B_{2r_i}(x)} \beta_\mu(z,r_j)^2 d \mu(z) \leq \delta_0 \, r_i^{n-1}.
\]
In particular, we can apply the Discrete Reifenberg Theorem~\ref{thm: discrete reifenberg} to $\mu$ and conclude. 
\end{proof}

%We now have the following 
%\[
%\begin{aligned}
% &\sum_{r_j \leq 2r_i} \int_{B_{2r_i}} \beta_i(z,r_j)^2 d \mu_i(z)  \\
% &= \sum_{r_j \leq 2r_i} \int_{B_{2r_i}} \beta_j(z,r_j)^2 d \mu_j(z)  \\
% &\leq c \sum_{r_j \leq 2r_i} \frac{1}{r_j^{n-1}}  \int_{B_{2r_i}(x)} \int_{B_{r_j}(z)}  \log\left(\frac{J(8r_i,y)}{J(r_i,y)}\right) d \mu_j(y) d \mu_j(z)  \\
% &\leq c \sum_{r_j \leq 2r_i} \frac{1}{r_j^{n-1}}  \int_{B_{2r_i+r_j}(x)} \int_{B_{r_j}(z)} \frac{\mu_j (B_{r_j}(y))}{r_j^{n-1}} \log\left(\frac{J(8r_i,y)}{J(r_i,y)}\right)  d \mu_j(y)  \\
% &\leq cM \int_{B_{4r_i}(x)} \sum_{r_j \leq 2 r_i}  \log\left(\frac{J(8r_i,y)}{J(r_i,y)}\right) \\
% &\leq CM \left(\sum_{p \in B_{4r_i}(x) \cap \text{spt} \mu_i} r_p^{n-1} \log\left(\frac{J(2,p)}{J(r_p)}\right) \right) \\
% &\leq CM \eta \mu_i(B_{4r_i}(x)) \\
% &\leq CM^2 \eta r_i^{n-1}. 
% \end{aligned}
%\]

\section{The Dichotomy}
In this section, we establish a key dichotomy: either all points in $\str_{\e, \eta}$ have small drop in the ACF monotonicity formula down to a small scale, or else all such points have a definite drop in their ACF formula at a smaller scale. This is simpler than the analogous dichotomy in other settings: a typical statement would say that either all points have small energy drop, or else the set of points with small energy drop looks lower dimensional in a quantitative sense. The reason behind this difference is that all blowup configurations for the ACF monotonicity formula have the same number of symmetries. The main form of the dichotomy is the following proposition:
\begin{proposition}\label{prop: key dichotomy}
	Fix an integer $n\geq 2$ and positive constants $\e , \bareta, \barrho, $ and $ \bar{J}_0$. There exists $\eta$ depending on these parameters such that the following holds. 
	Fix $r \in (0,2]$, 
	let $u, v :B_{10}(0)\to \R$ be nonnegative continuous functions satisfying \eqref{eqn: ACF setup} with $x_0 \in \Gamma^*\cap B_2(0)$
	 and  $\sup_{x \in B_{4r}(x_0)} J_x(4) \leq \bar J \leq \bar{J}_0$. Then at least one of the two possibilities occurs: 
	\begin{enumerate}
		\item For all $x \in \str_{\e, \eta r} \cap B_{2r}(x_0)$,  we have $J_x(\barrho\, r) > \bar J - \bareta$, or 
		\item For all $x \in \str_{\e, \eta r} \cap B_{2r}(x_0)$,   we have  $J_x (4 \eta\,r ) \leq  \bar J - \eta$.
	\end{enumerate}
\end{proposition}
Proposition~\ref{prop: key dichotomy} is a  direct consequence of the following lemma, which in turn is based on compactness and the quantitative estimates of Theorem~\ref{t:ACF} and Corollary~\ref{cor: coeff 1}.
\begin{lemma} \label{lemma: pre dichotomy}
	Fix $ n \geq 2$ and positive constants $\e , \bareta, \barrho$ and $ \bar{J}_0$. 
	There exist $\eta>0$ and $\tau>0$ depending on these parameters such that the following holds. 
	Suppose $u, v :B_{8}(0)\to \R$ are nonnegative continuous functions satisfying \eqref{eqn: ACF setup}   with $\sup_{x \in B_4(0)} J_x(4) \leq \bar J \leq \bar{J}_0$. 
	If there is some point $y \in B_2(0)$ such that
	\begin{equation}\label{eqn: tiny j drop1}
	 J_y(4\eta ) \geq \bar{J} - \eta,
	\end{equation}
then there is an affine hyperplane $L$ containing $y$ such that for every $x \in B_\tau(L) \cap B_2(0)$, we have
	\begin{equation}\label{eqn: definite drop}
 J_x( \barrho) > \bar{J} - \bareta \qquad\text{ and } \qquad \str_{\e, \eta } \cap B_2(0) \subset B_\tau(L).
	\end{equation}
	In particular, $J_x(\barrho) >\bar J -\bareta$ for all $x \in \str_{\e, \eta} \cap B_2(0)$.
\end{lemma}
\begin{proof}
	We argue by way of contradiction. Suppose the first condition in  \eqref{eqn: definite drop} fails. 
	We may thus find admissible pairs of functions $(u_j, v_{j})$ with $\sup_{x \in B_4(0)} J_x(4;\,u_{j}, \,v_{j}) \leq \bar {J}_j \leq \bar{J}_0$ and  points $y_{j} \in  B_2(0)$ satisfying \eqref{eqn: tiny j drop1} with $\eta_j \to 0$ in place of $\eta$, such that, for a sequence $\tau_j \to 0$ and for every affine hyperplane $L_j$ containing $y_j$,  there are points $x_j$ in $B_2(0) \cap B_{\tau_j}(L_j)$ such that 
		\begin{equation}
		\label{eqn: bad points1}
0 \leq	 J_{x_j}\left(\barrho ; \, u_{j}, \, v_{j}\right) \leq \bar J_j - \bareta\,
	\end{equation}
	Up to a subsequence, $\bar{J}_j \to \bar J \in [\bareta , \bar J_0]$. 
A basic calculation shows that $\log(J_{y_i}(1) / J_{y_i}(4\eta_j)) \leq 4 \eta_j/\bareta$. So, 
by Theorem~\ref{t:ACF}, for each $j$ there exist $\aone_{j}$, $\atwo_{j}>0$, and $\nu_{j} \in \mathbb{S}^{n-1}$ such that 
  \begin{equation}\label{eqn: apply tACFi1}
  \begin{aligned}
  	  \int_{B_4(y_{j} )\setminus B_{2\eta_j}(y_{j})} \left[u_{j} -\aone_{j} \left((x-y_{j})\cdot \nu_{j}\right)^+ \right]^2 &+ \left[v_{j} -\atwo_{j} ((x-y_{j})\cdot \nu_{j})^-\right]^2 \leq C\,\eta_j \int_{B_{4}(y_j)}( u_{j}^2 + v_j^2 ) 
  \end{aligned}
  \end{equation}
  for a constant $C$ depending on $n$ and $\bareta$.
Without loss of generality, we can multiply $u_{j}$ and $v_{j}$ by constants whose product is equal to $1$ so that $\aone_{j} =\atwo_{j}$, since this leaves $J$ unchanged. We may also assume that we have precomposed $u_{j}$ and $v_{j}$ with a rotation so that $\nu_j =e_n$ for all $j$. 

We want to show that $u_{j}$ and $v_{j}$ converge in $L^2$ to a pair of nondegenerate truncated linear functions, and so we must verify that the slopes $a_j$ do not degenerate to zero or blow up to infinity along the sequence. To this end, set $\hat{u}_{j} = u_{j}/a_j$ and $\hat{v}_{j} = v_{j}/a_j$. For $j$ sufficiently large, we have $4\eta_j \leq \max\{ \rho_0, \kappa_0^2/\bar{J}\}$, where $\rho_0$ and $\kappa_0$ are the constants from Corollary~\ref{cor: coeff 1}. So, applying Corollary~\ref{cor: coeff 1}, \eqref{eqn: apply tACFi1} becomes 
 \begin{equation*}\label{eqn: apply tACFi 2}
  \begin{aligned}
  	  \int_{B_4(y_{j} )\setminus B_{2\eta_j}(y_{j})} \left(\hat{u}_{j} -((x-y_{j})\cdot e_n)^+ \right)^2 &+ \left(\hat{v}_{j} - ((x- y_{j})\cdot e_n)^-\right)^2  \leq C\,\eta_j\,,
  \end{aligned}
  \end{equation*}
  for a constant $C$ depending on $n$ and $\bar \eta$.
  Up to a subsequence, 
  $y_j \to y_0\in \overline{B}_2(0)$. So, 
  letting $
 \hat{\ell}_1(x) =  ((x -y_0)\cdot e_n)^+$ and $\hat{\ell}_2(x) =  ((x -y_0)\cdot e_n)^-,$ we see that $\hat{u}_{j} \to \hat{\ell}_+$ and $\hat{v}_{j} \to \hat{\ell}_-$ in $L^2(B_4(y_0))$.
Corollary~\ref{cor: J convergence} then tells us that
\[
\lim_{j\to \infty} \,
\alpha_j^{-4}\, J_{y_0} (4; \, u_{j},\,  v_{j})
= \lim_{j\to \infty}\, J_{y_0} (4; \hat{u}_{j} , \hat{v}_{j}) = J_{y_0}(4; \hat{\ell}_+, \hat{\ell}_-) = c_*\,.
\]
On the other hand, by \eqref{eqn: tiny j drop1} and continuity, we see that $ 
\lim_{j \to \infty} \, J_{y_0}(4; \, u_{j} ,\, u_{j})  = \bar J$, and so it follows that $\lim_{j \to \infty} a_j = a_0:= (\bar{J}/c_*)^{1/4}>0$. 
Now,  set  $ \ell_+= a_0\, \hat{\ell}_+$ and $\ell_- =  a_0\, \hat{\ell}_-$, so
\begin{equation}	\label{eqn: l2 conv of rescaled}
	\int_{B_4(y_0) } (u_{j}- \ell_+)^2 + (v_{j} - \ell_-)^2 \to 0,
\end{equation}
and $J_x(r, \ell_+, \ell_-) = \bar{J}$ for each $r >0$ and $x \in \{ x\cdot e_n = 0\}$.

Now, we aim to reach a contradiction to \eqref{eqn: bad points1}.
For each $j$, choose $x_j$ corresponding to the affine hyperplane $L_j = \{ (x-y_j)\,\cdot\, e_n =0\}$. After passing to a subsequence, $x_j $ converges to a point $x_0 \in \overline{B}_2(0) \cap \{ (x-y_0)\,\cdot \,e_n=0\}$. Since $y_0 \in \overline{B}_2(0)$ as well, we have $B_{r}( x_0) \subset B_4(y_0)$ for any $r \leq 1$. So, thanks to Lemma~\ref{lem: improved convergence} once again, we have
\[
\lim_{j \to \infty} J_{x_j} (r; \, u_{j} , \, v_{j})\to J_{x_0}(r;\, \, \ell_+\,  \ell_-) = \bar{J}
\]
for all $r \in (0,1)$.
 On the other hand, by continuity and monotonicity, \eqref{eqn: bad points1} implies that
 \[
\lim_{j \to \infty} J_{x_j}(r; \, u_{j} , \, v_{j}) \leq \bar{J} - \bareta,
 \]
 for $r \in (0, \barrho)$,
leading us to a contradiction. This establishes the first part of \eqref{eqn: definite drop}.
\medskip

Next, let us prove that the second part of \eqref{eqn: definite drop} holds, with the same $\tau>0$ determined in the proof of \eqref{eqn: definite drop} above and up to further decreasing the parameter $\eta>0$ from the value determined above.
Once again, we argue by way of contradiction and suppose that the second part of \eqref{eqn: definite drop} fails for our fixed choices of $n,\e,\bareta, \barrho, \bar J_0$ and $\tau$. 
We may find  sequences $(u_{j}, v_{j}) , y_j$, and $\eta_j\to 0$ as above, satisfying the hypotheses of the lemma with $\eta_j$ in place of $\eta$, and a sequence of points $x_j$ violating the second part of \eqref{eqn: definite drop},  i.e. $x_j \in \str_{\e, \eta_j} \cap B_2(0) \setminus B_{\tau}(\{ x\cdot e_n = 0\})$.  Repeating the argument above, we find that \eqref{eqn: l2 conv of rescaled} holds and, up to a subsequence, $x_j \to x_0$ with $\text{dist}({x}_0, \{ x \cdot e_n =0\}) \geq \tau.$ Without loss of generality, we can assume that $x_0 \cdot e_n >0$ and thus $x_0\cdot e_n \geq \tau$. By continuity and the assumption that $ x_j \in \str_{\e, \eta_j}$ we have, for any fixed $r \in (0,1)$,
\[
\e \leq \lim_{j \to \infty} J_{x_j}(r\, ; \, u_{j},\,  v_{j}) = J_{x_0} (r\, ; \ell_+, \, \ell_-) \,.
\] 
On the other hand, $J_{x_0}(r\, ; \ell_+\,  , \ell_-) $  for $r \leq \tau/2$ since $\ell_-$ vanishes identically in $B_r(x_0)$, giving us a contradiction. We conclude that  the second part of \eqref{eqn: definite drop} holds.
\end{proof}

Proposition~\ref{prop: key dichotomy} is an easy consequence of Lemma~\ref{lemma: pre dichotomy}.
\begin{proof}[Proof of Proposition~\ref{prop: key dichotomy}]
Let $\eta$ be chosen according to Lemma~\ref{lemma: pre dichotomy}.
 By scaling, it suffices to prove the proposition with $r=1$ and $x_0=0$. Suppose we are not in case (2) in the statement of the proposition, so there is some $y \in \str_{\e ,\eta} \cap B_2(0)$ such that $J_y(4\eta ) > \bar J - \eta.$ Applying Lemma~\ref{lemma: pre dichotomy}, we then find that $J_x(\barrho) > \bar J  -\bar \eta $ for all $x \in \str_{\e,\eta} \cap B_2(0)$, i.e. we are in case (1).
 \end{proof}

The following direct corollary is how Proposition~\ref{prop: key dichotomy} will be applied in the next section.
 \begin{corollary}\label{cor: main app of key dichotomy}
 	Fix an integer $n\geq 2$ and positive constants $ \e, \bar{J}_0, \bareta,$ and $\barrho$. 
 	There exists $\eta$ depending on $n, \e, \bar J_0, \bareta,$ and $\barrho$ such that the following holds. 
 	Fix $R \in (0,1]$. For any $x \in B_2(0)$ and $r \in [\RR,2]$, 
 	let $u, v :B_{10}(0)\to \R$ be nonnegative continuous functions satisfying \eqref{eqn: ACF setup} with  $\sup_{x \in B_{4r}(x_0)} J_x(4) \leq \bar J \leq \bar{J}_0$. 
 	Then at least one of the two possibilities occurs:
 	\begin{enumerate}
 		\item For all $y \in \str_{\e,  \eta \RR}  \cap B_{2r}(x_0),$ we have $J_y(\barrho \,r )> \bar J -  \bareta$, or 
 		\item For all $y \in \str_{\e,  \eta \RR}  \cap B_{2r}(x_0),$ we have $J_y(4\eta r) \leq \bar J -\eta$. 
 		  	\end{enumerate} 
 \end{corollary}
 \begin{proof}
Choose $\eta>0$  according to Proposition~\ref{prop: key dichotomy}.  Suppose we are not in case (1) of the corollary. So, there is some $y \in \str_{\e,\eta \RR } \cap B_{2r}(x_0)$ with
$J_y(\barrho\, r ) \leq  \bar J -  \bareta.$
Since $r\geq \RR$ and  $\str_{\e, \eta \RR} \subset \str_{\e , \eta r}$, we see that $y \in \str_{\e, \eta r}$. Proposition~\ref{prop: key dichotomy} applied at scale $r$ then implies $J_y(4\eta\, r) \leq \bar J -\eta $ for all $y \in \str_{\e , \eta r} \cap B_{2r}(x_0)$, and thus for all  $y \in \str_{\e , \eta \RR} \cap B_{2r}(x_0)$ thanks again to the containment $\str_{\e, \eta \RR} \subset \str_{\e , \eta r}$. We are thus in case (2) of the corollary.
\end{proof}

\section{Quantitative estimates for $\str_\e$ and Rectifiability }
The main goal of this section is to prove Theorem~\ref{thm: main estimates} and Theorem~\ref{thm: main rectifiability}. In Section~\ref{ssec: packing down to scale R} we construct a good covering of $\str_{\e,\eta R}$ with estimates by balls of radius at least $R$, and in Sections~\ref{ss: proof of quant ests} and \ref{sec: rect} we prove Theorem~\ref{thm: main estimates} and \ref{thm: main rectifiability} respectively. 
In the previous section, we saw how Proposition~\ref{lemma: pre dichotomy}'s dichotomy took a simple form thanks to the fact that every ``cone'' for the Alt-Caffarelli-Friedman monotonicity formula is translationally invariant along some $(n-1)$-dimensional affine subspace. Thanks to this fact, in Section~\ref{ssec: packing down to scale R} we give a  covering construction that is substantially simpler than the covering used in many applications of the Naber-Valtorta machinery to singularity analysis. In particular, we can avoid entirely the ``good tree/bad tree construction'' and instead prove Lemma~\ref{lem: main packing} below with a fairly straightforward stopping time argument. 
\subsection{Main covering construction}\label{ssec: packing down to scale R}
This section is dedicated to Lemma~\ref{lem: main packing}, which is the main covering construction  that will be iterated to prove the quantitative estimates of Theorem~\ref{thm: main estimates}. Given $\e>0$ and $R \in (0,1]$, this lemma gives a covering, with estimates, of the $(\e, \eta R)$ stratum  for some  $\eta \ll 1$,  by balls of radii at least $R$.  The two key ingredients in the proof are the key dichotomy of Proposition~\ref{prop: key dichotomy} and the packing estimates of Proposition~\ref{lem: packing estimate for disjoint balls with small energy drop}.
The basic idea, which of course requires some technical modification, is the following. At each point in $x\in  \str_{\e,\eta R}\cap B_1(0)$, choose the smallest scale $r_x \in [R,1]$ such that $J_x(\rho \, r_x)$ stays uniformly large, and take a Vitali subcover of the corresponding cover. By design, the hypotheses of Proposition~\ref{lem: packing estimate for disjoint balls with small energy drop} are satisfied by this cover, and this gives us the packing estimates in part (2) of the lemma below.  Then, the dichotomy of the previous section gives us part (3):  either the drop in the ACF monotonicity formula stays small all the way down to scale $r_x=R$, or else the drop becomes large at some scale $r_x>R$ and we can apply Proposition~\ref{prop: key dichotomy}.

\begin{lemma}\label{lem: main packing}
	Fix $n \geq 2$, $\e>0$, and $\bar{J}_0>0$. 
	There are positive constants $C$ and $\eta$ depending on $n, \e, $ and $\bar{J}_0$ such that the following holds. 
	 Let $u, v :B_{10}(0)\to \R$ be nonnegative continuous functions satisfying \eqref{eqn: ACF setup} with $\bar{J} := \sup_{B_4(0)} J_x(4) \leq \bar J_0$,  and fix $\RR \in (0,1]$.
There is a  collection of balls $\{ B_{r_x}(x) \}_{x \in \mathcal{C}}$ with  $r_x \in [\RR, \eta]$ and $x \in  \str_{\e, \eta \RR}\cap B_1(0)$ satisfying the following properties:
	\begin{enumerate}
		\item The balls form a covering of the $(\e, \eta \RR)$ stratum in $B_1(0)$, that is, 
		\[
		\str_{\e ,\eta \RR} \cap B_1(0) \subset \bigcup_{x \in \mathcal{C}} B_{r_x}(x)\,.
		\]
		\item The balls satisfy the packing estimates 
		\[\sum_{x \in \mathcal{C} } r_x^{n-1}  \leq C.
		\]
		\item For every $x \in \mathcal{C}$, either $r_x= \RR$ or 
		\[
		\sup_{y \in B_{4r_x}(x) } J_y(4r_x) \leq \bar J -\eta.
		\] 
	\end{enumerate}
\end{lemma}
\begin{proof}
{\it Step 0:} Let us begin by fixing parameters. 
%
%Fix $\rho \leq c(n)$ where $c(n)\leq 1$ will be specified later. 
%
Let $\etaone= \etaone(n, \e )$ be chosen according to Proposition~\ref{lem: packing estimate for disjoint balls with small energy drop} and
 let $\barrho =  \etaone/10$.  
Choose $\eta$ according to the dichotomy of Corollary~\ref{cor: main app of key dichotomy}, depending on $n, \e, \bar J_0 ,\bareta, $ and $\barrho$ and thus on $n, \e, $ and $\bar J_0$. Up to possibly further decreasing $\eta$, we may assume that $\eta \leq \etaone$ and $\eta \leq 1/10$.  
\medskip

{\it Step 1:} We construct the collection of balls and show that it forms a cover of $\str_{\e, \eta\RR}\cap B_1(0)$. 
For each $x \in \str_{\e, \eta \RR}\cap  B_1(0)$, define the stopping time
\begin{equation}\label{eqn: stopping time}
	\hat{r}_x := \inf\left\{ r \in [R, 1] :  J_x( \barrho\,  r) > \bar J - \bareta \right\},
\end{equation}
with the convention that $\hat{r}_x =1$ if $J_x(\barrho ) \leq \bar{J} -\bar\eta$.
The collection $\{B_{\hat{r}_x/5}(x) \}_{x \in \str_{\e ,\eta \RR}\cap B_1(0)}$ is a cover of $\str_{\e, \eta \RR}\cap B_1(0)$. 
We apply the Vitali covering lemma to find a subset $\hat{\mathcal{C}}\subset \str_{\e, \eta \RR}$ such that the balls $\{B_{\hat{r}_x/5}(x) \}_{x \in \hat{\mathcal{C}}}$ are disjoint and the collection $\{B_{\hat{r}_x}(x) \}_{x \in \hat{\mathcal{C}}}$ forms a cover of   $\str_{\e, \eta \RR}$. 
We split this set of ball centers into three subsets:
\begin{align*}
	\mathcal{G} := \{ x \in \hat{\mathcal{C}}  \ : \ \hat{r}_x = R\} \,, \qquad \hat{\mathcal{A}} := \{ y \in \hat{\mathcal{C}}  \ : \ \hat{r}_y  \in(R,1)\}\, , \qquad \hat{\mathcal{V}} := \{ y \in \hat{\mathcal{C} } \ : \ \hat{r}_y = 1\}\,.
\end{align*}
For $x \in \mathcal{G}$, let $ r_x = \hat{r}_x = \RR$. The balls $B_{r_x}(x)$ for  $x \in \mathcal{G}$ will be included the final cover. The balls $B_{\hat{r}_y}(y)$ for $y \in \hat{\mathcal{A}} \cup\hat{\mathcal{V}}$ need to be further subdivided in the following simple way. For $y \in \hat{\mathcal{A}} \cup\hat{\mathcal{V}}$, set $r_y  =\max\{\RR, \,  \eta\, \hat{r}_y\} $ and  take a maximal $r_y/2$ disjoint set $\{x_{i,y}\}_{i=1}^{N_y}$ in $B_{\hat{r}_y}(y)$. 
There are at most $C_n\eta^{-n}$ such points, i.e. $N_y \leq C_n \eta^{-n}$,
 and the collection $\{ B_{r_y}(x_{i,y})\}_{i=1}^{N_y}$ is a cover of $B_{\hat{r}_y}(y)$.
  Let 
  \[
\mathcal{A} =\{ x_{i,y}  : y \in\hat{\mathcal{A}} , \, i=1, \cdots N_y \}, \qquad \quad \mathcal{V} =\{ x_{i,y}  : y \in\hat{\mathcal{V}} , \, i=1, \cdots N_y \}.
\]
If $x \in \mathcal{A}\cup \mathcal{V}$, then $x =x_{i,y}$ for at least one $y \in \hat{\mathcal{A}}\cup \hat{\mathcal{V}}$; set $r_x = r_y$ for the smallest such $r_y$. Now, let $\mathcal{C} = \mathcal{G} \cup \mathcal{ A} \cup \mathcal{V}$. The collection $
 \{B_{r_x}(x)\}_{x \in \mathcal{C}}$ is, by construction, a cover of $\str_{\e, \eta R} \cap B_1(0)$, so part (1) of the lemma holds.
\medskip

{\it Step 2:} We now verify the condition (2) of the lemma, using Proposition~\ref{lem: packing estimate for disjoint balls with small energy drop}'s packing estimate as the main tool. 
First, if $x \in \mathcal{G}$, then $J_x (\barrho \, r_x) \geq \bar J -\bar \eta $ by the stopping time definition and the continuity of the ACF formula with respect to $r$. 
Since we have chosen our parameters so that $\barrho \leq \etaone/10$, the  monotonicity of $J_x$ guarantees that $J_x(\etaone \, r_x/10 ) \geq \bar{J} - \etaone$. 
 The balls $\{ B_{r_x/5}(x)\}_{x \in \mathcal{G}}$ are disjoint by construction.
So, we may apply Proposition~\ref{lem: packing estimate for disjoint balls with small energy drop}, with $\RR/10$ in place of $\RR$, to $\{ B_{r_x/10}(x)\}_{x \in \mathcal{G}}$ to find that
\begin{equation}\label{eqn: packing G}
	\sum_{x \in \mathcal{G}} r_x^{n-1 }= \RR^{n-1 } \# (\mathcal{G})  \leq C(n). 
\end{equation} 

We prove the packing estimate for $x \in \mathcal{A}$ in a similar way. If $y \in \hat{\mathcal{A}}$, then
 by continuity $J(\barrho\, \hat{r}_y) = \bar{J}-\bareta$. 
Again by the choice of $\barrho$ and $\bareta$ and by monotonicity,  $J_y(\etaone\, \hat{r}_y /10) \geq \bar{J} - \etaone$, and the balls in the collection $\{B_{\hat{r}_y/5}(y) \}_{y \in \hat{\mathcal{A}}}$ are pairwise disjoint by construction. So, once more we apply Proposition~\ref{lem: packing estimate for disjoint balls with small energy drop} to $\{B_{\hat{r}_y/10}(y) \}_{y \in \hat{\mathcal{A}}}$ to find that 
	$ \sum_{y \in \hat{\mathcal{A}}} \hat{r}_y^{n-1 }  \leq C(n). $
Since $r_x \leq \hat{r}_y$ for any $x =x_{i,y} \in \mathcal{A}$ and 
$N_y \leq C(n) \eta^{-n}$ for all $y \in \hat{\mathcal{A}}$, 
this directly implies
\begin{equation}\label{eqn: packing A}
	\sum_{x \in \mathcal{A}} r_x^{n-1 } \leq C(n)\eta^{-n} \sum_{y \in \hat{\mathcal{A}}} \hat{r}_y^{n-1 }  \leq C(n,\eta) \,.
\end{equation}
Finally, the packing estimate for $\mathcal{V}$ is easy. 
The balls $\{B_{1/5}(x)\}_{x \in \hat{V} }$ are pairwise disjoint, so $\#\hat{\mathcal{V} }\leq C(n)$ and $\# \mathcal{V} \leq C(n) \eta^{-n}$.
Since $r_x = \eta$ for each $x \in \mathcal{V}$, this implies that  
\begin{equation}\label{eqn: packing V}
	\sum_{x \in \mathcal{V}} r_x^{n-1 } = \eta^{n-1} \#(\mathcal{V}) \leq  C(n,\eta).
\end{equation} 
Since $\eta$ depends on $n ,\e,$ and $\bar J_0$, together \eqref{eqn: packing G}, \eqref{eqn: packing G}, and \eqref{eqn: packing V} show that condition (2) of the lemma holds with a constant $C$ depending on $n, \e,$ and $\bar{J}_0$.
\medskip

{\it Step 3:}
Now we verify the third condition: either $r_x =R$ or there is a definite energy drop in all of $B_{4r_x}(x)$. 
The main tool is Corollary~\ref{cor: main app of key dichotomy}, which we recall states that for any $x \in B_2(0)$  and $r \in [\RR,2]$, at least one of 
\begin{align*}
	J_y(\barrho \,r )> \bar J -  \bareta&  \qquad \text{ for all } y \in \str_{\e,  \eta \RR}  \cap B_{2r}(x),\\
\text{ or }\ \ 	J_y(4\eta\, r) \leq \bar J -\eta & \qquad \text{ for all } y \in \str_{\e,  \eta \RR}  \cap B_{2r}(x)\,
\end{align*}
holds. 
If $x \in \mathcal{G}$, then $r_x = \RR$ and condition (3) of the lemmas holds. 
Next, as we observed in step 2,  if $y \in \hat{\mathcal{A}}$, then by the definition of $\hat{r}_y$ and continuity,  $J_y(\barrho \, \hat{r}_y) = \bar{J} - \bareta$. 
 By Corollary~\ref{cor: main app of key dichotomy} applied with $r = \hat{r}_y$, we have 
 \begin{equation}
 	\label{eqn: J drop on bigger ball}
 J_x( 4 \eta \, \hat{r}_y)  \leq \bar J - \eta \qquad \text{ for all } x \in \str_{\e, \eta \RR} \cap B_{2\hat{r}_y}(y)\,.
  \end{equation}
  In particular, for each $x \in \mathcal{A}$ that came from subdividing $B_{\hat{r}_y}(y)$, we have $B_{4r_x}(x) = B_{4\eta \hat{r}_y}(x) \subset B_{2 \hat{r}_y}(y).$ So, keeping in mind that $r_x \leq \eta \hat{r}_y$, \eqref{eqn: J drop on bigger ball} implies that 
  \[
		\sup_{z \in B_{4r_x}(x) } J_z(4r_x) \leq \bar J -\eta \qquad \text{ for all }x \in \mathcal{A}\,.
	\] 
	Thus condition (3) of the lemma holds for all $x \in \mathcal{A}$.
	Finally,  if $y \in \hat{\mathcal{V}}$, then $J_y(\barrho ) \leq \bar{J} - \bareta$ by the definition of $\hat{r}_y$. The exact same argument given for $\mathcal{A}$ above shows that 
	  \[
		\sup_{y \in B_{4r_x}(x) } J_y(4r_x) \leq \bar J -\eta \qquad \text{ for all } x \in \mathcal{V}.
	\] 
	So, condition (3) holds for all $x \in \mathcal{V}$.
This completes the proof of the lemma. 
\end{proof}

\subsection{Proof of the quantitative estimates}\label{ss: proof of quant ests}
We can now iteratively apply Lemma~\ref{lem: main packing} finitely many times to prove Theorem~\ref{thm: main estimates} in a standard way.
\begin{proof}[Proof of Theorem~\ref{thm: main estimates}]
First, let us note how the estimate \eqref{eqn: main cover conclusion} of Theorem~\ref{thm: main estimates} implies the two estimates in \eqref{eqn: eps r mink bound}. If \eqref{eqn: main cover conclusion} holds, then for every $x \in B_r\left(\Gamma^*_{\e, r} \cap B_1(0) \right)$, there is an index $i\in \{ 1,\dots, N\}$ such that $x \in B_r(x_i)$ and thus $B_r(x) \subset B_{2r}(x_i)$, with $N \leq C \, r^{1-n}$ as in \eqref{eqn: main cover conclusion}. So, 
\[
B_r\left(\Gamma^*_{\e, r} \cap B_1(0)  \right) \subset \bigcup_{i=1}^N B_{2r}(x_i)
\]
and in particular
$
|B_r\left(\Gamma^*_{\e, r} \cap B_1(0)  \right)| \leq N \omega_n (2r)^{n}  \leq C(n,\e, \bar{J}_0)\, r.
$
This proves the first estimate in  \eqref{eqn: eps r mink bound}. Next, since $\str_\e \subset \str_{\e,r}$, the estimate \eqref{eqn: main cover conclusion} gives the upper bound
$
\mathcal{H}^{n-1}_r (\str_{\e}) \leq C(n,\e,\bar{J}_0)
$ for each $r \in (0,1]$.  Passing $r\to 0$ establishes the Hausdorff measure bound in the second estimate of   \eqref{eqn: eps r mink bound}. \medskip

Now, fix $R \in (0,1]$. We will  construct a collection of balls satisfying \eqref{eqn: main cover conclusion} with $r=R$ by inductively applying Lemma~\ref{lem: main packing}. Let $\eta= \eta (n,\e, \bar{J}_0)$ be as in Lemma~\ref{lem: main packing}. If the covering $\{B_{r_x}(x)\}_{x \in \mathcal{C}}$ provided by Lemma~\ref{lem: main packing} has $r_x =\RR$ for every $x \in \mathcal{C}$, then \eqref{eqn: main cover conclusion} follows directly from the packing estimate (2), since $\str_{\e, \eta\, R} \subset \str_{\e, R}$. So, it suffices to construct a covering as in Lemma~\ref{lem: main packing} where in part (3), we always have $r_x = \RR$. The key observation is that the definite energy drop of (3) can only occur on $\lceil\, \bar J/ \eta \, \rceil$ scales by monotonicity, and so after $\lceil\, \bar J/\eta \, \rceil$ iterations of the lemma, we can only have $r_x = \RR$. 

More specifically, we claim that there exist a (finite) sequence of covers $\{B_{r_x}(x)\}_{x \in \mathcal{C}_i}$, whose centers are $\mathcal{C}_1, \mathcal{C}_2, \dots, $ satisfying the following properties:
\begin{enumerate}
	\item[($A_i$)] Covering: $\str_{\e, \eta \RR} \cap B_1(0) \subset \bigcup_{x \in \mathcal{C}_i} B_{r_x}(x),$
	\item[($B_i$)] Packing: $\sum_{x \in \mathcal{C}_i} r_x^{n-1} \leq C \left( 1 + \sum_{x \in \mathcal{C}_{i-1}  } r_x^{n-1}\right) $ for a contant $C=C(n,\e, \bar{J})$.
	\item[($C_i$)] Energy drop: for each $x \in \mathcal{C}_i$, we either have $r_x= R$ or $\sup_{B_{4r_x}(x)} J_y(4r_x) \leq \bar J - i \eta.$
	%\item $\sup_{x \in \mathcal{C}_i }r_x \leq \eta^i.$
\end{enumerate}

Observe $(C_i)$  guarantees that one of the collections $\{ B_R(x)\}_{x \in \mathcal{C}_i} $ for $i \in \{1 ,\dots , \lceil \eta /\bar{J} \rceil \}$ satisfies \eqref{eqn: main cover conclusion}, and so the theorem will follow directly from the claim.
\medskip

Lemma~\ref{lem: main packing} gives us such a covering in the base case $i=1$. 
 Suppose we have constructed a covering $\mathcal{C}_{i-1}$ satisfying $(A_{i-1}), (B_{i-1}),$ and $(C_{i-1})$.
  We construct $\mathcal{C}_i$ satisfying $(A_{i}), (B_{i}),$ and $(C_{i})$ in the following way. 
If $x \in \mathcal{C}_{i-1}$ has $r_x=R$, then we include it in $\mathcal{C}_i$. If $x \in \mathcal{C}_{i-1}$ has $r_x >R$, then we apply Lemma~\ref{lem: main packing} to $B_{r_x}(x)$.
 This gives us a collection of balls $\{B_{r_y}(y)\}_{y \in \mathcal{C}_{x, i}}$ such that 
 \[
 \str_{\e, \eta R} \cap B_{r_x}(x) \subset \bigcup_{y \in \mathcal{C}_{x,i}} B_{r_y}(y),
 \] 
 that satisfy the packing estimates 
 \begin{equation}\label{eqn: packing 100}
 	\sum_{y \in \mathcal{C}_{x,i}} r_y^{n-1} \leq C r_x^{n-1}
 	 \end{equation}
 	  for a constant $C= C(n,\e,\bar{J}_0)$. Moreover,  thanks to $(C_{i-1})$ and condition (3) of Lemma~\ref{lem: main packing}, for each $y \in \mathcal{C}_{x,i}$, either $r_y =\RR$, or else
 \begin{equation}
 	\label{eqn: energy drop final}
 \sup_{z \in B_{4r_y}}J_z(4r_y) \leq \sup_{B_{4r_x}(x)} J_z(4r_x) \leq \bar J  - (i-1)\eta - \eta =\bar{J} -i\eta . 
  \end{equation}
We let 
\[
\mathcal{C}_i = \{x \in \mathcal{C}_{i-1} : r_x =\RR\} \cup \bigcup_{x \in \mathcal{C}_{i-1}, r_x>\RR} \mathcal{C}_{x,i}.
\]
 By construction, $\{ B_{r_x}(x) \}_{x \in \mathcal{C}_i}$ is a cover of $\str_{\e, \eta R} \cap B_1(0)$, and so $(A_{i})$ holds, and 
 $(C_i)$ follows from \eqref{eqn: energy drop final} and the construction. 
  Finally, by $(B_{i-1})$  and \eqref{eqn: packing 100}, we have for a constant $C = C(n, \e ,\bar{J}_0)$, 
\begin{align*}
\sum_{x \in \mathcal{C}_i} r_x^{n-1}
 &= \sum_{\substack{x \in \mathcal{C}_{i-1} , \\ r_x = \RR}} r_x^{n-1} 
 + \sum_{\substack{x \in \mathcal{C}_{i-1}, \\ r_x >\RR}} \Big( \sum_{y \in \mathcal{C}_{x,i} }r_y^{n-1} \Big) \\
 & \leq  
 \sum_{\substack{x \in \mathcal{C}_{i-1} , \\ r_x = \RR}} r_x^{n-1} 
 + C \sum_{\substack{x \in \mathcal{C}_{i-1}, \\ r_x >\RR}}  r_x^{n-1}
 \leq C \sum_{x \in \mathcal{C}_{i-1}} r_x^{n-1}\,.
\end{align*}
 Thus $(B_i)$ holds as well.  This proves the inductive step, and thus concludes the proof of the theorem.
\end{proof}

\subsection{Rectifiability}\label{sec: rect}
We now prove Theorem~\ref{thm: main rectifiability}, which says that $\Gamma^*\cap B_1(0)$ is $\cH^{n-1}$-rectifiable. 
This proof is a standard step in the Naber-Valtorta framework:
 %and is reminiscent of the proof of Proposition~\ref{lem: packing estimate for disjoint balls with small energy drop}:
  combining the upper Ahlfors regularity of $\mathcal{H}^{n-1}\mres \str_\e$ shown in Theorem~\ref{thm: main estimates} and the $L^2$ subspace approximation of Theorem~\ref{t:l2est2}, we show the hypotheses of Naber-Valtorta's Rectifiable Reifenberg Theorem~\ref{thm: rect reif} hold, and thus $\Gamma^*_\e$ is $\cH^{n-1}$-rectifiable.
In fact, instead of arguing on $\str_\e$ directly,  we will break each $\Gamma^*_\e$ into countably many smaller pieces, each of which we will show is rectifiable.  Since the countable union of rectifiable sets is again rectifiable, it follows that $\str$ is rectifiable as well. 
\begin{proof}[Proof of Theorem~\ref{thm: main rectifiability}]
Let $\bar J_0 = \sup_{x\in B_4(0)} J_x(4)$. Fix any $\e>0$.  Let $\sigma = \sigma(n,\e, \bar J_0) >0$ be a small fixed number to be specified later in the proof, and let 
\[
S =S_{\e, \sigma}= \{ x \in B_1(0) : J(0^+) \in \left[\e, \e \exp\{\sigma \}\right]\,\} \subset \str_\e \cap B_1(0)\,.
\] 
We will prove that $S$ is $\mathcal{H}^{n-1}$-rectifiable. 
To this end, we first claim that for all $x_0 \in S$, there is a scale $r>0$ depending on $x_0$ such that
\begin{equation}
	\label{eqn: tiny log by contradition}
\log \left(\frac{ J_y(8r) }{J_y(0^+)}\right) \leq 2 \sigma \qquad \text{ for all } y \in S \cap B_r(x_0).
\end{equation}
Suppose by way of  contradiction that we can find a sequence of points $y_j \in S$ with $y_j \to x_0$ and scales $r_j \to 0$ for which $\log ({ J_{y_j}(8r_j) }/{J_{y_j}(0^+)}) >2 \sigma$.
In other words, exponentiating and multiplying through by $J_{y_j}(0^+)$, 
\[
 J_{y_j}(8r_j) > J_{y_j}(0^+) \exp\{2 \sigma \} \geq \e \exp\{2 \sigma\}\, ,
\]
where the final inequality holds because $S\subset \Gamma^*_\e$. Taking the $\limsup$ of both sides as $j \to \infty$,  we deduce from the upper semicontinuity property \eqref{eqn: USC} that $J_{x_0}(0^+) \geq  \e \exp\{2\sigma\}$. On the other hand $J_{x_0}(0^+) \leq \e \exp\{\sigma\}$ by virtue of being in $S$. We reach a contradiction and see that \eqref{eqn: tiny log by contradition} holds for $r>0$ sufficently small.\medskip

To prove that $S$ is $\mathcal{H}^{n-1}$-rectifiable, it suffices to show that $S \cap B_r(x_0)$ is rectifiable for any $x_0 \in S$ and for $r>0$ such that \eqref{eqn: tiny log by contradition} holds. Let $\mu= \mathcal{H}^{n-1}\mres S.$ 
Applying Theorem~\ref{t:l2est2}, we then see that, 
letting $r_j= 2^{-j}$,
\begin{align}\label{eqn: rect pf first one}
	\sum_{r_j \leq r} \int_{B_r(x_0)} \beta_\mu(z ,r_j)^2 \, d \mu(z) 
	&\leq  \sum_{r_j \leq r} \ \frac{C}{r_j^{n-1}}\int_{B_r(x_0)} \int_{B_{r_j}(z)} \log \left(\frac{ J_y(8r_j) }{J_y(r_j)}\right) \,d\mu(y) \, d\mu(z)
	\end{align}
Now applying  Fubini's theorem, for any $r_j \leq r$.
\begin{align*}
	\int_{B_r(x_0)} \int_{B_{r_j}(z)} \log \left(\frac{ J_y(8r_j) }{J_y(r_j)}\right) \,d\mu(y) \, d\mu(z)  
	%	& =\int_{\R^n}  \int_{\R^n} 1_{\{|y-z|<r_j\}} 1_{\{|z-x_0|<r\}} \log \left(\frac{ J_y(8r_j) }{J_y(r_j)}\right) \,d\mu(y) \, d\mu(z) \\
%	& =\int_{\R^n} \log \left(\frac{ J_y(8r_j) }{J_y(r_j)}\right) \int_{\R^n} 1_{\{|y-z|<r_j\}} 1_{\{|z-x_0|<r\}} \ \,d\mu(z) \, d\mu(y) \\
		& =\int_{\R^n} \log \left(\frac{ J_y(8r_j) }{J_y(r_j)}\right) \mu(B_{r_j}(y)\cap B_{r}(x_0)) \, d\mu(y) \\
&\leq   C\int_{B_{2r}(x_0)} \log \left(\frac{ J_y(8r_j) }{J_y(r_j)}\right) \mu(B_{r_j}(y)) \, d\mu(y)\,.
\end{align*}
Next, by the
Ahlfors upper bound in \eqref{eqn: eps r mink bound} in Theorem~\ref{thm: main estimates}, we know that $
 \mu(B_s(y) ) \leq C(n,\e, \bar J_0) s^{n-1}$
  for all $y \in B_{2r}(x_0)$ and $s \leq 2r.$ 
This together with \eqref{eqn: rect pf first one} implies
\begin{align*}
		\sum_{r_j \leq r} \int_{B_r(x_0)} \beta_\mu(z ,r_j)^2 \, d \mu(z)  &{\leq} \sum_{r_j \leq r}  \ C\int_{B_{2r}(x_0)} \log \left(\frac{ J_y(8r_j) }{J_y(r_j)}\right) \, d\mu(y) \\
	&	=  C\int_{B_{2r}(x_0)} \sum_{r_j \leq r} \log \left(\frac{ J_y(8r_j) }{J_y(r_j)}\right) \, d\mu(y) \\
	& =  C\int_{B_{2r}(x_0)}  \log \left(\frac{ J_y(8r) }{J_y(0^+)}\right) \, d\mu(y) \leq C \sigma \, r^{n-1}\,.
\end{align*}
In the final line we used \eqref{eqn: tiny log by contradition} and the Ahlfors upper bound in \eqref{eqn: eps r mink bound} once again.
Now, let us choose $\sigma>0$ small enough so that $C\sigma<\delta_{0}$, where $\delta_0$ is the dimensional constant from the rectifiable Reifenberg theorem, Theorem~\ref{thm: rect reif}. So, by Theorem~\ref{thm: rect reif}, $S\cap B_r(x_0)$ is rectifiable.
\end{proof}

\section{Uniqueness of Function Blowups}

This section is dedicated to the proof of Theorem~\ref{thm: uniqueness of blowups}. The basic idea is the following. Theorems~\ref{thm: main rectifiability} and  Theorem~\ref{thm: main estimates} imply that $\mathcal{H}^{n-1}\mres \str_\e$ is a Radon measure and has an approximate tangent plane at $\mathcal{H}^{n-1}$-a.e. point. Together with the differentiation theory for measures, we prove the existence of limits of $\Delta u^{x,r}$ and $\Delta v^{x,r}$ as distributions as $r \to 0$. Using the quantitative form of the Alt-Caffarelli-Friedman monotonicity formula of Theorem~\ref{t:ACF}, we directly relate these distributional limits to the truncated linear functions arising as the blowup limits of $u^{x,r_k}$ and $v^{x,r_k}$ along a sequence $r_k \to 0$. The independence of the former from the sequence $r_k \to 0$ allows us to prove the uniqueness of the latter.

%\begin{theorem}
%	For $\cH^{n-1}$-a.e. $x \in \Gamma^*$, there exist two numbers $\aone^x, \atwo^x > 0$ and a unit vector $\nu^x \in S^{n-1}$ such that the rescaled functions
%	\[
%		u^{x, r}(z) = \frac{u(x + r z)}{r} \qquad v^{x, r}(z) = \frac{v(x + r z)}{r}
%	\]
%	converge locally in (strong) $L^2$ topology to $\ell_+(z) = \alpha_1^x (z \cdot \nu)_+$ and $\ell_-(z) = \alpha_2^x (z \cdot \nu)_-$, respectively. Moreover, $\alpha_1^x\alpha_2^x = c(n)J_x(0^+)$, where $c(n) > 0$ is a dimensional constant.
%\end{theorem}

\begin{proof}[Proof of Theorem~\ref{thm: uniqueness of blowups}]
{\it Step 1:}
	Fix $\epsilon > 0$. We will prove the theorem under the assumption that $x \in \Gamma^*_\epsilon \cap B_{1/2}(0)$. As the statement is purely qualitative and $\Gamma^* = \cup_\epsilon \Gamma^*_\epsilon$, this will imply the conclusion.
	
By Theorems \ref{thm: main rectifiability} and \ref{thm: main estimates}, $\nu := \cH^{n-1}\mres \Gamma^*_\epsilon \cap B_1(0)$ is a Radon measure, and for $\cH^{n-1}$-a.e. $x \in \Gamma^*_\epsilon \cap B_1(0)$, 
\begin{equation}\label{eqn: density and tangent measure}
	\lim_{r \searrow 0}\ \frac{\nu(B_r(x))}{\omega_{n-1}r^{n-1}} = 1 \qquad \text{ and } \qquad \nu^{x, r}(E)  \overset{*}{\rightharpoonup }\cH^{n-1}\mres \{z \cdot \nu^x = 0\}\ \text{ for some } \nu^x \in S^{n-1}\,.
\end{equation}
Here we set $\nu^{x, r}(E) := \frac{\nu(x + r E)}{r^{n-1}}$ for $r>0$ and convergence is in the weak-$*$ topology for measures as $r \to 0$. See \cite[Theorem 10.2]{MaggiBook} for proofs of these properties of $\cH^{n-1}$-rectifiable sets with locally finite $\cH^{n-1}$ measure.

The distributional Laplacian $\Delta u$ (defined by acting on $\phi \in C^\infty_c(\R^n)$ by $\Delta u(\phi) = \int u \Delta \phi$) is positive in the sense that $\Delta u(\phi) \geq 0$ for any nonnegative test function $\phi \in C^\infty_c(\R^n)$. It is easy to show that any such distribution is a bounded linear functional on $C^0_c(\R^n)$, so by the Riesz representation theorem (see for instance \cite[Theorem 4.7]{MaggiBook}), we may express $\Delta u(\phi) = \int \phi\, d\mu_u$ for a Radon measure $\mu_u$. In the same way, the distributional Laplacian $\Delta v$ of $v$ is identified with a Radon measure $\mu_v.$ 

Applying the Lebesgue-Besicovitch differentiation theorem (see for instance \cite[Theorem 5.8]{MaggiBook}) to $\mu_u$ (resp. $\mu_v$) and $\nu$, 
 we see that for $\mathcal{H}^{n-1}$-a.e. $x \in \Gamma^*_\e$, 
 the limits
 \begin{equation}\label{eqn: densities exist}
 	 \zeta_u(x) := \lim_{r \searrow 0} \frac{\mu_u(B_r(x))}{\nu(B_r(x))}, \qquad \zeta_v(x) := \lim_{r \searrow 0} \frac{\mu_v(B_r(x))}{\nu(B_r(x))}\qquad  
  \text{  exist and are finite}\,,
  \end{equation}
and we may write $\mu_u= \zeta_u d\nu + \mu_u^s$ and $\mu_v= \zeta_v d\nu + \mu_v^s$, where $\mu_u^s$ (resp. $\mu_v^s$) and $\nu$ are mutually singular.
Let us restrict our attention, then, to those $x\in \str_\e$ such that both \eqref{eqn: density and tangent measure} and \eqref{eqn: densities exist} hold. For any such $x$, we have 
\begin{equation}
	\label{eqn: conv lap meas}\begin{split}
\mu_u^{x, r}(E) := \frac{\mu_u(x + r E)}{r} & \overset{*}{\rightharpoonup}\zeta_u(x) \,\cH^{n-1}\mres \{z \cdot \nu^x = 0\}\,,\\
 \mu_v^{x, r}(E) := \frac{\mu_v(x + r E)}{r} & \overset{*}{\rightharpoonup}\zeta_v(x) \, \cH^{n-1}\mres \{z \cdot \nu^x = 0 \}\,
 	\end{split}
\end{equation}
 in the weak-$*$ topology for measures.  Notice that the rescaled measures correspond to the distributional Laplacians of the rescaled functions $u^{x,r}$ and $v^{x,r}$, i.e.  $\mu_u^{x, r} = \Delta u^{x, r}$ and $\mu_v^{x, r} = \Delta v^{x, r}$.\medskip
%Together
%The set $\{\limsup_{r \searrow 0} \frac{\mu_i(B_r(x))}{r^{n-1}} = \infty \}$ has $\cH^{n-1}$ measure $0$.

 %We now make several observations about $\mu_i$ and $\nu$ of a purely measure-theoretic nature:
%	\begin{enumerate}
	%	\item For $\cH^{n-1}$-a.e. $x \in \Gamma^*_\epsilon$, we have that $\lim_{r \searrow 0} \frac{\nu(B_r(x))}{\omega_{n-1}r^{n-1}} = 1$, and that there exists a $\nu^x \in S^{n-1}$ such that $\nu^{x, r}(E) = \frac{\nu(x + r E)}{r^{n-1}}$ converges to $\cH^{n-1}\mres \{z \cdot \nu^x = 0\}$ in the weak-$*$ topology.
		%\item The set $\{\limsup_{r \searrow 0} \frac{\mu_i(B_r(x))}{r^{n-1}} = \infty \}$ has $\cH^{n-1}$ measure $0$.
%		\item For $\cH^{n-1}$-a.e. $x \in \Gamma^*_\epsilon$, the limit $\zeta_i(x) = \lim_{r \searrow 0} \frac{\mu_i(B_r(x))}{\nu(B_r(x))}$ exists and is finite.
	%	\item We may write $\mu_i = \zeta_i d\nu + \mu_i^s$, where $\mu_i^s$ and $\nu$ are mutually singular.
%	\end{enumerate}
	
	%Let us restrict our attention, then, to $x$ for which property (1) holds, property (2) does not occur, and the limit in (3) exists (for both $i =1, 2$). 
	%At any such point, 

{\it Step 2:} We claim  that 
	\begin{equation}
		\label{eqn: unif bounded in L2}
	\limsup_{r \searrow 0} \left\{  \|u^{x, r}\|_{L^2({B_1})}  +   \|v^{x, r}\|_{L^2({B_1})}\right\} < \infty.
		\end{equation}
	 To this end, set
	\[
		q_{u, r} = \max\{\|u^{x, r}\|_{L^2(B_1)}, 1\}, \qquad q_{v, r} = \max\{\|v^{x, r}\|_{L^2(B_1)}, 1\}.
	\] 
%	We now consider $u^{x, r}$ and $v^{x, r}$ in more detail, starting with the observation that 
Since $x \in \str_\e$, we have  $\epsilon \leq J_x(0^+) = \lim_{r \searrow 0} J_x(r)$, and so $\lim_{r\searrow 0} \log \frac{J_x(r)}{J_x(0^+)} = 0$. 
	So, we can apply the stability inequality, Theorem \ref{t:ACF}, to $\hat{u}_r := u^{x, r}/q_{u, r}$  and $\hat{v}_r : =v^{x, r}/q_{v, r}$ (as usual, noting that dividing by positive constants leaves the ratio $J_0(1)/J_0(0^+)$ unchanged) to obtain
	\begin{equation}\label{e:stabrenorm}
		\int_{B_1}\left| \hat{u}_r - \aone^{r} (z \cdot \nu^{r})^+\right|^2 + \int_{B_1}\left| \hat{v}_r - \atwo^{r} (z \cdot \nu^{r})^-\right|^2 \leq C \log \frac{J_x(r)}{J_x(0^+)}
	\end{equation}
	for positive constants $\aone^{r},\atwo^{r}$ and vectors $\nu^{r}\in \mathbb{S}^{n-1}$. (These depend on $x$ as well as $r$, but we suppress the dependence in the notation since $x$ is fixed throughout this step). Clearly $\aone^{r}, \atwo^{r} < C$ uniformly for all $r$ sufficiently small.  
%	From the Caccioppoli inequality we have that
%	\[
%		\left\|\frac{u_i^{x, r}}{q_{i, r}}\right\|_{H^1(B_{1/2})} \leq C.
%	\]
	Assume by way of contradiction that $q_{u, r_k} \rightarrow \infty$ along a sequence $r_k \searrow 0$. Up to a  subsequence,  $\aone^{r_k} \rightarrow \aone^* \in [0, C]$, $\atwo^{r_k} \rightarrow \atwo^* \in [0, C]$, and $\nu^{r_k} \rightarrow \nu^*\in S^{n-1}$, and thus  \eqref{e:stabrenorm} tells us that $\hat{u}_{r_k} \to \hat{u}^*:= \aone^*(z \cdot \nu^*)^+$ and 
	$\hat{v}_{r_k} \to \atwo^*(z \cdot \nu^*)^-$ in $L^2(B_1)$.
%	 After passing to a subsequence, which we do not relabel, the sequences ${u^{x, r}}/{q_{u, r_k}} $ and ${v^{x, r}}/{q_{v, r_k}}$  converge weakly in $H^1(B_{1/2})$ and strongly in $L^2(B_{1/2})$ to functions $u^*, v^* \in H^1(B_{1/2})$ respectively. 
	% Up to further subsequences, we may also assume that $\aone^{x, r_k} \rightarrow \aone^* \in [0, C]$, $\atwo^{x, r_k} \rightarrow \atwo^* \in [0, C]$, and $\nu^{x, r_k} \rightarrow \nu^*$; from \eqref{e:stabrenorm}, we see that $u^* = \aone^* (x \cdot \nu^*)^+$ and $v^* = \atwo^* (x \cdot \nu^*)^-$ and the sequences of functions converge in $L^2(B_1)$.
	Moreover, $\aone^* > 0$ since $\|\hat{u}_{r_k}\|_{L^2(B_1)} = 1$ for all $k$ large enough.
%	$\aone^{x, r_k} > c > 0$ uniformly for $k$ large enough, as $\|u^{x, r_k}/q_{u, r_k}\|_{L^2(B_1)} = 1$ and $\aone^{x, r_k}(x \cdot \nu^{x, r_k})$ is arbitrarily close to this function in $L^2$; this implies that $\aone^* > 0$.
	
	From \eqref{eqn: densities exist} we have that $\Delta u^{x, r_k}(B_1) \rightarrow \omega_{n-1}\, \zeta_u(x)< \infty$ and  $\Delta v^{x, r_k}(B_1) \rightarrow \omega_{n-1} \,\zeta_v(x) < \infty$. Since  $q_{u, r}, q_{v, r} \geq 1$, we must have 
	\[
	\limsup_{r \searrow 0}\Delta
	% \frac{u^{x, r}}{q_{u, r}}
	\hat{u}_{r}(B_1) \leq \omega_{n-1}\zeta_u(x) < \infty
	\quad 
	\text{ and }
	\quad
	\limsup_{r \searrow 0}\Delta \hat{v}^{x, r}(B_1) \leq \zeta_v(x) < \infty\,.
	\] 
	Up to passing to a further subsequence, this implies that $\Delta \hat{u}_{r_k} \rightarrow \Delta \hat{u}^*$ in the sense of distributions on $B_{1/2}$. Thanks to \eqref{eqn: conv lap meas} and the definition of $\hat{u}_r$, we have
	\[
	\Delta \hat{u}^* = \left( \lim_{k \to \infty} \frac{1}{q_{u, r_k}} \right) \zeta_u(x) \cH^{n-1}\mres\{z \cdot \nu^x = 0 \}.
	\]
	 This implies, in particular, that $\Delta \hat{u}^* = 0$, i.e. that $\hat{u}^*$ is harmonic on $B_{1/2}$. But this clearly contradicts the expression $\hat{u}^* = \aone^* (z \cdot \nu^*)_+$ with $\aone^* > 0$ obtained above. This proves \eqref{eqn: unif bounded in L2}.
\\

{\it Step 3:}
	With this in mind, we may repeat the blowup argument in Step 2 above except this time without normalizing by $q_{u, r}$,  $q_{v, r}$. 
	This shows that, along every sequence of radii $r_k \searrow 0$, there exists a subsequence with $u^{x, r_k} \rightarrow u^* := \aone^*(z\cdot \nu^*)^+$ and $v^{x, r_k} \rightarrow v^* :=\atwo^*(z\cdot \nu^*)^-$ in $L^2(B_{1})$,
	 %with $v_i = \alpha_i^* (z \cdot \nu^*)_\pm$ 
	 for some $\aone^*, \atwo^* \in [0,C] $ and $\nu^*\in \mathbb{S}^{n-1}$, and moreover $\Delta u^{x, r_k} \rightarrow \Delta u^*$ and $\Delta v^{x, r_k} \rightarrow \Delta v^*$ as distributions. 
	 To complete the proof of the theorem, we must show that $\aone^*, \atwo^*$, and $\nu^*$ do not depend on the sequence $r_k$. 	 The convergence \eqref{eqn: conv lap meas}  implies that $\aone^* = \zeta_u(x)$ and $\atwo^*= \zeta_v(x)$,  and that $\nu^*$ is either $\nu^x$ or $- \nu^x$ (but either is possible).  We must, then, show that $\nu^*$ is one or the other independent of subsequence. Before doing so, however, note that a direct application of Corollary \ref{cor: J convergence} implies that $c_* \zeta_u(x) \zeta_v(x) = J_x(0^+)$, so in particular both $\zeta_u(x)$ and $\zeta_v(x)$ are nonzero.
	
	To show that $\nu^*$ is independent of subsequence, observe that as $\pm \nu^x $ are the only two limit points for $\nu^{x, r}$, for any $\delta > 0$, there exists an $r_0 > 0$ such that  $\min \{|\nu^{x, r} - \nu^x|, |\nu^{x, r} + \nu^x|\} < \delta$ for all $r < r_0$. If $\nu^{x, r}$ has more than one limit point, then there must be a sequence $r_k \searrow 0$ and another sequence $s_k < r_k$ with $1 - \frac{s_k}{r_k} \rightarrow 0$ such that $|\nu^{x, r_k} - \nu^x | < \delta$ and $|\nu^{x, s_k} + \nu^x| < \delta$. The first of these implies that
	\[
		\limsup_{k \rightarrow \infty} \int_{B_1} |u_1^{x, r_k} - \zeta_1(x) (x \cdot \nu^x)_+|^2 \leq \limsup_{k \rightarrow \infty} \int_{B_1} |u_1^{x, r_k} - \zeta_1(x) (x \cdot \nu^{x, r_k})_+|^2 + C \delta \leq C \delta.
	\]
	Restricting to $B_{s_k/r_k}$ and rescaling,
	\[
		\int_{B_1} |u_1^{x, s_k} - \zeta_1(x) (x \cdot \nu^x)_+|^2 \leq C \delta \left(\frac{r_k}{s_k}\right)^{n-2} \leq C \delta.
	\]
	On the other hand, $\nu^{x, s_k}$ is close to $- \nu^x$, so
	\[
		\int_{B_1} |u_1^{x, s_k} + \zeta_1(x) (x \cdot \nu^x)_+|^2 \leq C \delta.
	\]
	This implies that $\zeta_1(x) \leq C \sqrt{\delta}$ for any $\delta > 0$, which contradicts that $\zeta_1(x) > 0$. Finally, thanks to Lemma~\ref{lem: improved convergence}, the $L^2$ convergence immediately implies $W^{1,2}$ convergence.
\end{proof}

%
%\appendix
%\section{Things that are true but not needed}
%

\bibliographystyle{alpha}
\bibliography{RectQACFref.bib}

\end{document}